\documentclass[a4paper,11pt]{amsart}
\usepackage{amsmath,amsthm,amssymb,enumerate}
\usepackage{graphicx}

\numberwithin{equation}{section}
\numberwithin{figure}{section}

\newtheorem{thm}{Theorem}[section]
\newtheorem{lem}[thm]{Lemma}
\newtheorem{prop}[thm]{Proposition}
\newtheorem{cor}[thm]{Corollary}
\newtheorem{definition}[thm]{{Definition}}
\newtheorem{remark}[thm]{{Remark}}
\newtheorem{example}[thm]{\rm{Example}}

\newenvironment{df}{\begin{definition}\rm}{\end{definition}}
\newenvironment{rem}{\begin{remark}\rm}{\end{remark}}
\newenvironment{ex}{\begin{example}\rm}{\end{example}}

\title{Unknotting numbers and triple point cancelling numbers of torus-covering knots}
\author{Inasa Nakamura}

\address{
Department of Mathematics, 
Gakushuin University, 
1-5-1 Mejiro Toshima-ku Tokyo, 171-8588 JAPAN} 

\email{inasa@math.gakushuin.ac.jp}
\date{}
\begin{document}

\begin{abstract}
It is known that any surface knot can be transformed to an unknotted surface knot or a surface knot which has a diagram with no triple points by a finite number of 1-handle additions. The minimum number of such 1-handles is called the unknotting number or the triple point cancelling number, respectively. 
In this paper, we give upper bounds and lower bounds of unknotting numbers and triple point cancelling numbers of torus-covering knots, 
which are surface knots in the form of coverings over the standard torus $T$. 
Upper bounds are given by using $m$-charts on $T$ presenting torus-covering knots, and lower bounds are given by using quandle colorings and quandle cocycle invariants.

\end{abstract}
\maketitle
\section{Introduction}
A {\it surface link (knot)} is the image of a smooth embedding of a (connected) closed surface into the Euclidean 4-space $\mathbb{R}^4$. In this paper, we assume that a surface link is oriented.  A surface link is said to be {\it unknotted} if it is equivalent to the boundary of a disjoint union of handlebodies \cite{HoKa79}. 
For a surface link $F$, a {\it 1-handle} is an oriented 3-ball $H=D^2 \times I$ embedded in $\mathbb{R}^4$ such that $H \cap F=(D^2 \times \partial I) \cap F$ and the orientations coincide for $H\cap F$, where $D^2$ is a 2-disk and $I$ is an interval. By a {\it 1-handle addition} we obtain a new surface link $F^\prime$ from $F$ by
\[
F^\prime=(F-(\mathrm{Int}D^2 \times \partial I)) \cup (\partial D^2 \times I), 
\]
where we give $F^\prime$ the orientation induced from $F$ (see \cite{Boyle88, HoKa79}).
It is known \cite{HoKa79} (see also \cite{Kamada99}) that for any surface link $F$, there is a sequence of 1-handle additions such that the result is an unknotted surface link or a surface link which has a diagram with no triple points called {\it pseudo-ribbon} \cite{Kawauchi02}. The minimum number of such 1-handle additions is called the {\it unknotting number} denoted by $u(F)$ or the {\it triple point cancelling number} denoted by $\tau(F)$, respectively. Since an unknotted surface link is pseudo-ribbon, it is obvious that 
\[
0 \leq \tau(F) \leq u(F).
\]
Kamada \cite{Kamada99} gave an upper bound of the unknotting number by using a graphical method called an $m$-chart presenting a surface link. Iwakiri \cite{Iwakiri} gave a lower bound of the unknotting number or the triple point cancelling number by using quandle colorings and quandle cocycle invariants. 
For other results on unknotting numbers or triple point cancelling numbers, see \cite{HoKa79, HMS79, Kanenobu96, KM97, Kawauchi02, Miyazaki86, Satoh04}.
In this paper, we consider for $F$ a \lq\lq torus-covering $T^2$-link", which is a surface link in the form of an unbranched covering over the standard torus $T$ \cite{N}. A  torus-covering $T^2$-link is determined from a pair of commutative $m$-braids $a$ and $b$, which is denoted by $\mathcal{S}_m(a,b)$.
The aim of this paper is to give an upper bound and a lower bound of the unknotting number or the triple point cancelling number of $\mathcal{S}_m(a,b)$ in certain particular form, by using the methods based on \cite{Iwakiri, Kamada99}.
 
Let us denote by $\Delta$ an $m$-braid with a full twist. 
\begin{thm} \label{thm:0503-01}
For any $m$-braid $b$ and any integer $n$, 
\[
u(\mathcal{S}_m(b, \Delta^n)) \leq m-1. 
\]
\end{thm}

From now on, we consider torus-covering knots, i.e. torus-covering links with one component except in Theorem \ref{thm:0510-3}. 
Let $p$ be an odd prime. For an $m$-braid $b$, we denote by $\hat{b}$ the closure of $b$, and we use the same notation $\hat{b}$ for a diagram of $\hat{b}$. For the dihedral quandle $R_p$ of order $p$, we call an $R_p$-coloring a {\it $p$-coloring} \cite{CJKLS, Fox} (see also \cite{Iwakiri}). Let us denote by $\mathrm{Col}_p(\hat{b})$ the set of $p$-colorings for $\hat{b}$, and we denote the number of its elements by $|\mathrm{Col}_p(\hat{b})|$; see Section \ref{sec:4}.
It is known \cite{Iwakiri} that $\mathrm{Col}_p(\hat{b})$ is a linear space isomorphic to $(\mathbb{Z}/p\mathbb{Z})^k$ and hence $|\mathrm{Col}_p(\hat{b})|=p^k$ for some positive integer $k$ (see Lemma \ref{lem:linear}). 
 
\begin{thm}\label{thm:unknotting}
Put $l=2$ (resp. $p$) if $m$ is odd (resp. even).  
Let $b$ be an $m$-braid such that $\hat{b}$ is a knot. 
Let $k$ be a positive integer such that 
$|\mathrm{Col}_p(\hat{b})|=p^k$. 
Then, for any integer $n$,
\[
u(\mathcal{S}_m(b, \Delta^{ln})) \geq k-1. 
\]
\end{thm}

In some cases, we can determine the unknotting numbers. 
Let $\sigma_1, \ldots, \sigma_{m-1}$ be the standard generators of the $m$-braid group.
\begin{thm}\label{thm:=}
Let $b$ be an $m$-braid presented by an element of the group generated by $\sigma_i^p$ ($i =1,\ldots,m-1)$ such that $\hat{b}$ is a knot.
 Then, for any integer $n$,
\[
u(\mathcal{S}_m(b, \Delta^{ln})) =m-1.
\]
\end{thm}

Let us consider the shadow quandle cocycle invariant of $\hat{b}$ associated with Mochizuki's 3-cocycle of $R_p$, and denote it by $\Psi_p^*(\hat{b})$; see Section \ref{sec:4}. We regard $\Psi_p^*(\hat{b})$ as a multi-set consisting of elements of $\mathbb{Z}/p\mathbb{Z}$, where repetitions of the same elements are allowed. Let us denote by 
$a_0(\Psi_p^*(\hat{b}))$ the number of $0$ in the multi-set $\Psi_p^*(\hat{b})$ i.e. the number of $p$-colorings which contribute $0$ in $\Psi_p^*(\hat{b})$, where $0$ denotes the identity element of $\mathbb{Z}/p\mathbb{Z}$.  

\begin{thm}\label{thm:0510-3}
Let $m$ be a positive odd integer with $m \not \equiv 0 \pmod{p}$. 
Let $b$ be an $m$-braid.
Let $k$ be a positive integer such that 
$|Col_p(\hat{b})|=p^k$, and let $k^\prime$ be a positive integer such that 
$a_0(\Psi^*(\hat{b}))<p^{k^\prime}$.
Then, for an integer $n$ with $n \not \equiv 0 \pmod{p}$, 
\[
\tau (\mathcal{S}_m(b, \Delta^{2n})) \geq k-k^\prime+2.
\]
\end{thm}

In some cases, we obtain a better estimate (see Remark \ref{rem:1}).
For an $m$-braid $b$ as in Theorem \ref{thm:=}, 
let $\nu_i$ ($i=1,\ldots,m-1$) be the image of the presentation of $b$ by the homomorphism $f: \langle \sigma_1^p, \ldots, \sigma_{m-1}^p \rangle \to \mathbb{Z}/p\mathbb{Z}$ defined by $f(\sigma_j^p)=1 \pmod{p}$ if $j=i$ and otherwise zero, where $j \in \{1,\ldots,m-1\}$.
\begin{thm}\label{thm:0510-1}
Let $m$ be a positive odd integer with $m \not\equiv 0 \pmod{p}$. 
Let $b$ be an $m$-braid as in Theorem \ref{thm:=} such that $\nu_i \in \{x^2 \mid x \in \mathbb{Z}/p\mathbb{Z}\}-\{0\}$ ($i =1,\ldots,m-1$). 
Then, for an integer $n$ with $n \not\equiv 0 \pmod{p}$, 
\[
\tau(\mathcal{S}_m(b, \Delta^{2n})) \geq \frac{m-1}{2}.
\]
\end{thm}

Let $g: (\mathbb{Z}/p\mathbb{Z})^{m-1} \to \mathbb{Z}/p\mathbb{Z}$ be a map defined by 
$g(y_1, \ldots,y_{m-1})=\sum_{i=1}^{m-1} \nu_i y_i^2$. Let $U_j$ be the set of elements of $(\mathbb{Z}/p\mathbb{Z})^{m-1}$ with exactly $j$ non-zero entries ($j=1,\ldots,m-1$). Let $p^\prime$ be the minimum number of $j \in \{1,\ldots,m-1\}$ satisfying  $0 \in g(U_j)$; note that for $m \geq 4$, if $\nu_i \neq 0 \pmod{p}$ $(i=1,\ldots,m-1$), then $p^\prime=2$ or $3$; see Lemma \ref{lem:0502-2}.

\begin{thm}\label{cor:0510-2}
Let $p$ be an odd prime greater than three. 
Let $b$ be a $3$-braid as in Theorem \ref{thm:=} such that $\nu_i \neq 0$ ($i =1,2$).  
If $p^\prime \neq 2$, then, for an integer $n$ with $n \not\equiv 0 \pmod{p}$,  
\[
\tau(\mathcal{S}_3(b, \Delta^{2n}))=u(\mathcal{S}_3(b, \Delta^{2n}))=2.
\]
\end{thm}

\begin{cor}\label{cor:0122-01}
Let $p>3$ be a prime congruent to $3 \pmod{4}$. 
Let $b$ be a $3$-braid as in Theorem \ref{thm:0510-1}.  
Then, for an integer $n$ with $n \not\equiv 0 \pmod{p}$,  
\[
\tau(\mathcal{S}_3(b, \Delta^{2n}))=u(\mathcal{S}_3(b, \Delta^{2n}))=2.
\]
\end{cor}

This paper is organized as follows. 
In Section \ref{section1}, we review torus-covering links. 
In Section \ref{sec:3}, we review the notion of an $m$-chart on a standard torus $T$ presenting a torus-covering link. 
In Section \ref{unknotting-number}, we prove Theorem \ref{thm:0503-01} by using $m$-charts on $T$. 
In Section \ref{sec:4}, we review quandle colorings and quandle cocycle invariants.
In Section \ref{sec:6}, we prove Theorems \ref{thm:unknotting} and \ref{thm:=}.
In Section \ref{sec:7}, we calculate quandle cocycle invariants associated with Mochizuki's 3-cocycle and we show Theorems \ref{thm:0510-3}, \ref{thm:0510-1} and \ref{cor:0510-2}, and Corollary \ref{cor:0122-01}.
\section{Torus-covering links} \label{section1}
%
Let $T$ (resp. $S^2$) be a standard torus (resp. 2-sphere), i.e. the boundary of an unknotted solid torus (resp. a 3-ball) in $\mathbb{R}^3 \times \{0\} \subset \mathbb{R}^4$.
 It is known \cite{Kamada94, Kamada02} that any surface link can be presented in the form of the closure of a surface braid. We can modify the terms as follows: Any surface link can be presented in the form of a braided surface over $S^2$ \cite{Kamada92, N, Rudolph}. 
Torus-covering links were introduced in \cite{N} by considering the standard torus $T$ instead of the standard 2-sphere $S^2$.
In this section, we review the definition of a torus-covering link by using the notion of a braided surface over $T$, which is a modified notion of a braided surface over a 2-disk \cite{Kamada92, Rudolph}; see \cite{N}. A torus-covering $T^2$-link is uniquely determined from a pair of commutative $m$-braids $a$ and $b$, which we call basis braids, and we denote by $\mathcal{S}_m(a,b)$ a torus-covering $T^2$-link with basis $m$-braids $a$ and $b$ \cite{N}.
  
We work in the smooth category, and we assume that embeddings are locally flat. 
A {\it surface link} is the image of a smooth embedding of a closed surface into 
$\mathbb{R}^4$. 
Two surface links are said to be {\it equivalent} if one is taken to the other 
by an orientation-preserving self-diffeomorphism of $\mathbb{R}^4$. Let $D^2$ be a 2-disk, and let $m$ be a positive integer. 
  
\begin{df}
A closed surface $F$ embedded in $D^2 \times T$ is called a {\it braided surface over $T$} of degree $m$ if $\pi |_{F} \,:\, F \rightarrow T$ is a branched covering map of degree $m$, where $\pi \,:\, D^2 \times T\to T$ is the projection to the second factor. 
A braided surface $F$ is called \textit{simple} if $\#(F \cap \pi^{-1}(x))=m-1$ or $m$ for each $x \in T$. 
Take a base point $x_0$ of $T$. 
Two braided surfaces over $T$ of degree $m$ are {\it equivalent} if there is a fiber-preserving ambient isotopy of $D^2 \times T$ rel $\pi^{-1}(x_0)$ which carries one to the other. 
\end{df}

Let $N(T)$ be a tubular neighborhood of $T$ in $\mathbb{R}^4$. We identify $N(T)$ with $D^2 \times T$. 
\begin{df} \label{Def2-1}
A {\it torus-covering link} is a surface link in $\mathbb{R}^4$ 
presented by a simple braided surface over $T$, where we regard the braided surface as in 
$N(T) \subset \mathbb{R}^4$. 
 \end{df}
 
A {\it torus-covering $T^2$-link} is a torus-covering link each of whose components is of genus one. 
Let us consider a torus-covering $T^2$-link $F$. Let us fix a point $x_0$ of $T$, and take a meridian $\mu$ and a longitude $\lambda$ of $T$ with the base point $x_0$. 
A {\it meridian} is an oriented simple closed curve on $T$ which bounds the 2-disk of the solid torus whose boundary is $T$ and which is not null-homologous in $T$. A {\it longitude} is an oriented simple closed curve on $T$ which is null-homologous in the complement of the solid torus in the three space $\mathbb{R}^3 \times \{0\}$ and which is not null-homologous in $T$. Since $F$ is an unbranched cover over $T$, the intersections $F \cap \pi^{-1}(\mu)$ and $F \cap \pi^{-1}(\lambda)$ are closures of classical braids. 
Cutting open the solid tori at the 2-disk $\pi^{-1}(x_0)$, we obtain a pair of classical braids. We call them {\it basis braids} \cite{N}. The basis braids of a torus-covering $T^2$-link are commutative, and for any commutative $m$-braids $a$ and $b$, there exists a unique torus-covering $T^2$-link with basis braids $a$ and $b$ \cite{N}. For commutative $m$-braids $a$ and $b$, we denote by $\mathcal{S}_m(a,b)$ the torus-covering $T^2$-link with basis $m$-braids $a$ and $b$ \cite{N}. 
   
\section{Graphs called $m$-Charts on $T$}\label{sec:3}
The notion of an $m$-chart on a 2-disk was introduced in \cite{Kamada92, Kamada02} to present a simple surface braid. 
By regarding an $m$-chart on a 2-disk as drawn on $S^2$, it presents a simple braided surface over $S^2$ \cite{Kamada92, Kamada02, N}. 
Two simple braided surfaces over $S^2$ of the same degree are equivalent if and only if the presenting $m$-charts are C-move equivalent \cite{Kamada92, Kamada96, Kamada02}.
These notions can be modified for $m$-charts on a standard torus $T$ \cite{N}. An $m$-chart on $T$ presents a torus-covering link \cite{N}. 
 \\
 
We identify $D^2$ with $I \times I$, where $I=[0,1]$. 
When a simple braided surface $F$ over $T$ is given, we obtain a graph on $T$, as follows. 
Consider the singular set $\mathrm{Sing}(\pi^\prime(F))$ of the image of $F$ by the projection $\pi^\prime$ to $I \times T$. Perturbing $F$ if necessary, we can assume that $\mathrm{Sing}(\pi^\prime(F))$ consists of double point curves, triple points, and branch points. Moreover we can assume that the singular set of the image of $\mathrm{Sing}(\pi^\prime(F))$ by the projection to $T$ consists of a finite number of double points such that the preimages belong to double point curves of $\mathrm{Sing}(\pi^\prime(F))$. 
Thus the image of $\mathrm{Sing}(\pi^\prime(F))$ by the projection to $T$ forms a finite graph $\Gamma$ on $T$ such that the degree of its vertex is either $1$, $4$ or $6$. 
An edge of $\Gamma$ corresponds to a double point curve, and a vertex of degree $1$ (resp. $6$) corresponds to a branch point (resp. a triple point). 

For such a graph $\Gamma$ obtained from a simple braided surface $F$, we give orientations and labels to the edges of $\Gamma$, as follows. 
Let us consider a path $\rho$ in $T$ such that $\Gamma \cap \rho$ is a point $P$ of an edge $e$ of $\Gamma$. 
Then $F \cap \pi^{-1} (\rho)$ is a classical $m$-braid with one crossing in $\pi^{-1}(\rho)$ such that $P$ corresponds to the crossing of the $m$-braid. Let $\sigma_{i}^{\epsilon}$ ($i \in \{1,\ldots, m-1\}$, $\epsilon \in \{+1, -1\}$) be the presentation of $F \cap \pi^{-1}(\rho)$. 
Then label the edge $e$ by $i$, and moreover give $e$ an orientation such that the normal vector of $\rho$ coincides (resp. does not coincide) with the orientation of $e$ if $\epsilon=+1$ (resp. $-1$). We call such an oriented and labeled graph an {\it $m$-chart of $F$}. 
\\
 
 In general, we define an $m$-chart on $T$ as follows. 

\begin{df}
A finite graph $\Gamma$ on $T$ is called an {\it m-chart on $T$} if it satisfies the following conditions: 

\begin{enumerate}
 \item Every edge is oriented and labeled by an element of $\{1,2, \ldots, m-1\}$. 
 \item Every vertex has degree $1$, $4$, or $6$.
 \item  The adjacent edges around each vertex are oriented and labeled as shown in Fig. \ref{Fig1-1}, where we depict a vertex of degree 1 by a black vertex, and a vertex of degree 6 by a white vertex, and we call a vertex of degree 4 a crossing.
 \end{enumerate}
 \end{df}
 
 \begin{figure}
 \includegraphics*{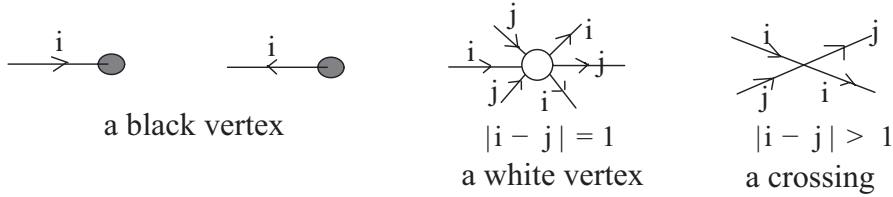}
 \caption{Vertices in an $m$-chart.}
 \label{Fig1-1}
 \end{figure}
  
 When an $m$-chart $\Gamma$ on $T$ is given, we can reconstruct a simple braided surface $F$ over $T$ such that the original $m$-chart $\Gamma$ is an $m$-chart of $F$; see \cite{Carter-Saito, Kamada92-2, Kamada02, N}. 
A black vertex (resp. a white vertex) of $\Gamma$ presents a branch point (resp. a triple point) of $F$. Let $\rho$ be an oriented path. If $\Gamma \cap \rho$ does not contain black vertices of $\Gamma$, then it presents an $m$-braid. 
A {\it singular $m$-braid} is an $m$-braid allowed to intersect transversely in finitely many double points called {\it singular points}. 
If $\Gamma \cap \rho$ consists of a black vertex of $\Gamma$, it presents a singular braid with one singular point and with no crossings. In this paper, we denote the singular braid $\Gamma \cap \rho$ by $\dot{\sigma_i}^{\epsilon}$ ($i \in \{1,\ldots,m-1\}, \epsilon\in\{+1,-1\}$), where $\sigma_i^{\epsilon}$ is the presentation of the $m$-braid presented by $\Gamma \cap \rho^\prime$, where $\rho^\prime$ is an oriented path parallel to $\rho$ and sufficiently near $\rho$ such that $\Gamma \cap \rho^\prime$ consists of an inner point of the edge extending from the black vertex. See \cite{Kamada92-2, Kamada02}. 
\\

Two $m$-charts on $T$ are {\it C-move equivalent} if they are related by a finite sequence of ambient isotopies of $T$ rel $\pi^{-1}(x_0)$ and CI, CII, CIII-moves shown in Fig. \ref{cmove} \cite{N}; see \cite{Kamada02} for the complete set of CI-moves. 
Two simple braided surfaces over $T$ of degree $m$ are equivalent if and only if $m$-charts of them are C-move equivalent \cite{N}; see also \cite{Kamada92, Kamada92-2, Kamada02} .  
Hence it follows that for two $m$-charts on $T$, their presenting torus-covering links are equivalent if the $m$-charts are C-move equivalent.  
\begin{figure}
 \includegraphics*{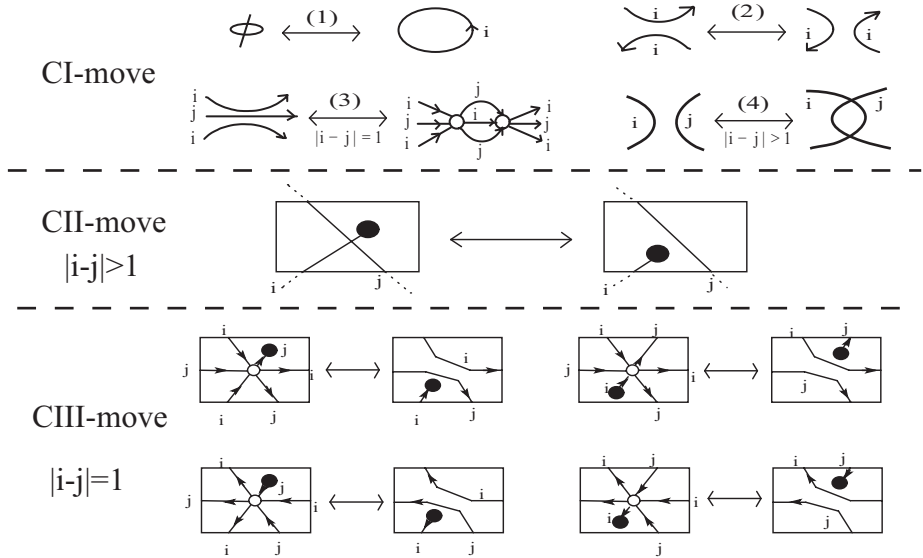}
\caption{CI, CII, CIII-moves. For CI-moves, we give only several examples.} 
\label{cmove}
 \end{figure} 

A torus-covering link is presented by an $m$-chart on $T$. In particular, a torus-covering $T^2$-link $\mathcal{S}_m(a,b)$ is presented by an $m$-chart on $T$ without black vertices \cite{N}. In this paper, we denote by $\Gamma_m(a,b)$ an $m$-chart as illustrated in Fig. \ref{2013-0121-01}, presenting $\mathcal{S}_m(a,b)$. 

\begin{figure}
 \includegraphics*{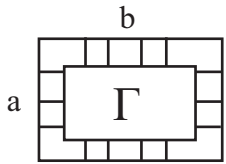}
\caption{An $m$-chart $\Gamma_m(a,b)$ presenting $\mathcal{S}_m(a,b)$, where $\Gamma$ contains no black vertices.} 
\label{2013-0121-01}
 \end{figure} 

In an $m$-chart, a {\it free edge} is an edge whose endpoints are a pair of black vertices \cite{Kamada02}. In this paper, we denote by $F_i$ a free edge labeled with $i$ ($i \in \{1,\ldots,m-1\}$).
For an $m$-chart $\Gamma$ on $T$, let $D_0$ be a disk in $T$ such that $D_0$ contains the base point $x_0$ and further $\Gamma \cap D_0=\emptyset$. 
Let $\cup_i F_i$ be a split union of free edges in $D_0$ and let us denote by $\Gamma \cup \cup_i F_i$ the union of $\Gamma$ in $T$ and $\cup_i F_i$ in $D_0$.

\section{Upper bounds} \label{unknotting-number}

%
For an $m$-chart $\Gamma$, we can consider a new $m$-chart obtained from  
inserting a free edge to $\Gamma$ in such a way that it is disjoint with $\Gamma$, which we call the $m$-chart obtained from $\Gamma$ by {\it insertion of a free edge} \cite{Kamada92-2}.
For a surface link $F$ and an $m$-chart $\Gamma$ presenting $F$,  insertion of a free edge to $\Gamma$ presents a 1-handle addition to $F$ \cite{Kamada92-2, Kamada02}. The notion of an unknotted $m$-chart was introduced in \cite{Kamada92} to present an unknotted surface link: An $m$-chart on $S^2$ is called {\it unknotted} if it consists of free edges, and an unknotted $m$-chart presents an unknotted surface link \cite{Kamada92}; here we regard an $m$-chart on a 2-disk as drawn on $S^2$.
By using these terms, Kamada \cite{Kamada99} investigated $m$-charts on $S^2$, and an upper bound of the unknotting number was given: For an $m$-chart $\Gamma$ on $S^2$, 
$u(\Gamma) \leq w(\Gamma) +m-1$, and hence it follows from $u(F) \leq u(\Gamma)$ that $u(F) \leq w(\Gamma) +m-1$, where $u(\Gamma)$ is the {\it unknotting number of $\Gamma$} i.e. the minimum number of free edges whose insertions are necessary to transform $\Gamma$ to be C-move equivalent to an unknotted $m$-chart, and $w(\Gamma)$ is the number of white vertices in $\Gamma$, and $F$ is the surface link presented by $\Gamma$ \cite{Kamada99}. 
In this section, we show similar results for $m$-charts on $T$ presenting torus-covering links.
We give an upper bound of the unknotting number of a torus-covering link, by using $m$-charts on $T$ and insertion of free edges. 

In Section \ref{sec:4-1}, we define an unknotted $m$-chart on $T$ and we give Theorem \ref{Prop5-5}, by using which we show Theorem \ref{thm:0503-01}. 
In Section \ref{sec:4-2}, we show Theorem \ref{Prop5-5}.

\subsection{Upper bounds of unknotting numbers}\label{sec:4-1}

 \begin{df}
We say an $m$-chart on $T$ is {\it unknotted} 
if it consists of free edges.
\end{df}

  \begin{lem} \label{0302-2}
An unknotted $m$-chart on $T$ presents a torus-covering link which is unknotted. 
  \end{lem}

\begin{proof}
Let $\Gamma$ be an unknotted $m$-chart on $T$.
Let $F$ be the torus-covering link presented by $\Gamma$. 
  Let $F^\prime$ be a surface link obtained by $\Gamma$ by regarding it as an $m$-chart on $S^2$. 
  Then $F$ is obtained from $F^\prime$ by additions of $m$ 1-handles in the following form:
\begin{enumerate}
\item
$F^\prime \subset \mathbb{R}^3 \times (\infty,0]$ such that $F\cap (\mathbb{R}^3 \times \{0\})$ consists of $m$ copies of a 2-disk $D$. 
\item
The 1-handles are $m$ copies of  $H=D^\prime \times I$ embedded in $\mathbb{R}^3 \times [0, \infty)$, where $D^\prime$ is a 2-disk and $I$ is an unknotted semicircle such that $I \cap (\mathbb{R}^3 \times \{0\})=\partial I$ and $H \cap F^\prime=D^\prime \times \partial I$ is a pair of 2-disks in $D$.
\end{enumerate}
  Since $F^\prime$ is unknotted \cite{Kamada92, Kamada02}, $F^\prime$ is the boundary of a disjoint union of handlebodies  which we denote by $K$. Since $K \cup H$ is also a disjoint union of handlebodies such that the boundary is $F$, it follows that $F$ is also unknotted. 
  \end{proof}

For an $m$-chart $\Gamma$ on $T$, let us call the minimum number of free edges necessary to transform $\Gamma$ by their insertion to be C-move equivalent to an unknotted $m$-chart the {\it unknotting number of an $m$-chart on $T$}, and denote it by $u(\Gamma)$.

  \begin{prop}\label{prop:1}
  For a torus-covering link $F$ presented by an $m$-chart $\Gamma$ on $T$, 
\[
u(F) \leq u(\Gamma). 
\]
  \end{prop}
\begin{proof}
Since insertion of a free edge to $\Gamma$ presents a 1-handle addition to $F$ \cite{Kamada92-2, Kamada02}, together with  Lemma \ref{0302-2}, we have the result.
\end{proof}

  \begin{prop} \label{lem1}
Let $\Gamma$ be an $m$-chart on $T$, and 
let $w(\Gamma)$ be the number of white vertices in $\Gamma$. 
Then
\[
u(\Gamma) \leq w(\Gamma)+m-1. 
\]
\end{prop}

\begin{proof}
We use the same argument used in \cite{Kamada99}, as follows.
The $m$-chart $\Gamma$ on $T$ can be transformed to have no white vertices by inserting a free edge to near each white vertex and applying a CI-move of type (2) (see the figure denoted by (2) in Fig. \ref{cmove}) and a CIII-move, as in Fig. \ref{fig-0516-01}. 
We denote the resulting $m$-chart by $\Gamma^\prime$. 
The vertices of $\Gamma^\prime$ are black vertices or crossings. Let us consider the union of connected edges with the same label. It is homeomorphic to an interval or a circle. We call it an edge or a loop respectively. If it is an edge, then its endpoints are a pair of black vertices. By applying CII-moves, it is deformed to a free edge. Hence we can assume that $\Gamma^\prime$ consists of free edges and loops. By applying a CI-move of type (2) and then an ambient isotopies of $T$ and by applying CII-moves if necessary (see Fig. \ref{fig-0513-02}), from the loop nearest  $\cup_{i=1}^{m-1}F_i$, we can see that $\Gamma^\prime \cup \cup_{i=1}^{m-1}F_i$ is C-move equivalent to an unknotted $m$-chart on $T$. 
This implies that the unknotting number $u(\Gamma)$ is at most $w(\Gamma)+m-1$.
\begin{figure}
\includegraphics*{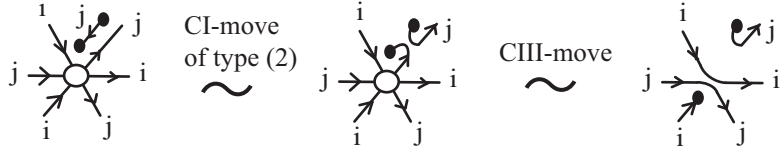}
\caption{We can eliminate a white vertex by inserting a free edge $F_j$. In the figure, $|i-j|=1$.}
\label{fig-0516-01}
\end{figure}
\begin{figure}
\includegraphics*{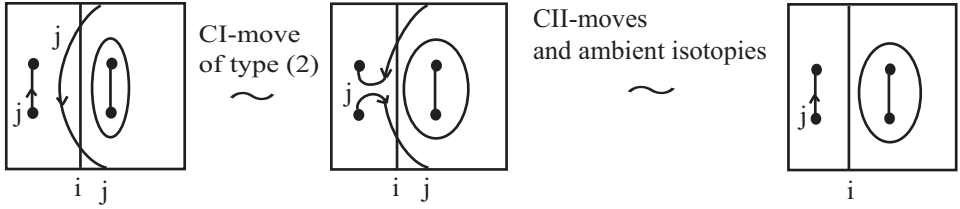}
\caption{We can eliminate a loop with label $j$ by inserting a free edge $F_j$. In the figure, $|i-j|>1$.}
\label{fig-0513-02}
\end{figure}
\end{proof}

In some cases, we can give a better upper bound of $u(\Gamma)$ than 
Proposition \ref{lem1}. 
Recall that $\Delta$ is an m-braid
with a full twist. Let $\Delta=(\sigma_1 \sigma_2 \cdots \sigma_{m-1})^m$ be the presentation.

\begin{thm}\label{Prop5-5}
For any $m$-braid $b$ and any integer $n$, 
\[
u(\Gamma_m(b, \Delta^n)) \leq m-1. 
\]
\end{thm}

\begin{proof}[Proof of Theorem \ref{thm:0503-01}]
The torus-covering link $\mathcal{S}_m(b, \Delta^n)$ is presented by 
an $m$-chart $\Gamma_m(b,\Delta^n)$.
By Proposition \ref{prop:1} and Theorem \ref{Prop5-5}, we have the result.

\end{proof}

\subsection{Proof of Theorem \ref{Prop5-5}}\label{sec:4-2}
 \subsubsection{Key Lemma}
\begin{lem}\label{lem2}
For an $m$-braid $b$, let us fix a word presentation of $b$, and let $l(b)$ be the word length of $b$, i.e. the number of letters of $b$. 
Let $b=\sigma_j^\epsilon b_1$ ($j \in \{ 1,2,\ldots,m-1 \}$, $\epsilon \in \{+1, -1\}$) be the presentation, where $b_1$ is the subword of $b$ with the word length $l(b)-1$. 
Then 
\[
\Gamma_m(b, \Delta^n) \cup \cup_{i=1}^{m-1} F_i \sim \Gamma_m(b_1, \Delta^n) \cup \cup_{i=1}^{m-1} F_i, 
\]
where \lq\lq $\sim$" denotes C-move equivalence.
\end{lem}
\subsubsection{Proof of Theorem \ref{Prop5-5}}
\begin{proof}

Put $\Gamma=\Gamma_m(b,\Delta^n)$. We show that $\Gamma^\prime=\Gamma \cup \cup_{i=1}^{m-1} F_i$ is C-move equivalent to $\cup_{i=1}^{m-1} F_i$. 
By Lemma \ref{lem2}, by induction on the number of letters of $b$, $\Gamma^\prime$ is C-move equivalent to $\Gamma_m(e, \Delta^n) \cup \cup_{i=1}^{m-1} F_i$, 
where $e$ denotes the empty word. 
It is known \cite{Kamada92, Kamada02} that for an intersection of an $m$-chart and a 2-disk containing no black vertices, by CI-moves, we can rewrite the $m$-chart in the 2-disk as we like as long as it has no black vertices.
Since $\Gamma_m(e, \Delta^n)$ contains no black vertices, it is C-move equivalent to an $m$-chart consisting of parallel loops presenting $\Delta^n$. 
Hence, by applying CI-moves of type (2), we can transform it to $\cup_{i=1}^{m-1} F_i$ (see Fig. \ref{fig-0513-02}); thus $u(\Gamma) \leq m-1$.
\end{proof}

\subsubsection{Proof of Lemma \ref{lem2}}
 
\begin{proof}
It suffices to show that 
\[
\Gamma_m(b, \Delta^n) \cup F_j \sim \Gamma_m(b_1, \Delta^n) \cup F_j.
\]
Let us consider the case when $\epsilon=+1$; then we have $b=\sigma_j b_1$.
Put $\Gamma=\Gamma_m(b, \Delta^n)$, and 
put $\Gamma_0=\Gamma \cup F_j$. 
Let $D$ be a disk in $T-(\mu \cup \lambda)$ such that $\Gamma \cap(T-D)$ contains no vertices.
Let us take a point $y_0 \in \partial D$ such that 
there exists an oriented path $\eta$ from $x_0$ to $y_0$ with $\Gamma \cap \eta =\emptyset$ and further a conjugate by $\eta$ of a loop $\partial D$ oriented anti-clockwise with a base point $y_0$ is homotopic to $\mu \lambda \mu^{-1} \lambda^{-1}$. 
See Fig. \ref{fig-0510-02}. 

Let $e$ be the edge of $\Gamma_0 \cap (T-D)$ presenting $\sigma_j$ the first letter of $b$.
Move the free edge $F_j \subset \Gamma_0$ by an ambient isotopy in such a way that it is parallel to $e$ and the orientation is reverse. Then let us apply a CI-move of type (2) to these edges. By an ambient isotopy of $T$, we can deform the resulting $m$-chart to the $m$-chart $\Gamma_1$ obtained from $\Gamma-e$ by attaching two black vertices to $\partial e \in \partial D$; note that the attached  black vertices come from $F_j$. We denote them by $v_1$ and $v_2$ in such a way that  the connected edge is oriented from $v_1$ (resp. toward $v_2$). 

Let $\rho$ be an oriented loop such that it starts from $x_0$, goes along $\eta$ to $y_0$, goes along $\partial D$ clockwise to near $v_2$, goes around $v_2$ anti-clockwise and comes back to $x_0$ by the reverse course. Then $\Gamma_1\cap \rho$ presents the braid 
$\Delta^{n} \sigma_j \Delta^{-n}$.
By Lemma \ref{lem:0522-1}, we can deform $\Gamma_1$ by C-moves and ambient isotopies to an $m$-chart $\Gamma_2$ satisfying that there is exactly one black vertex $v_2^\prime$ such that we can take an oriented loop $\rho^\prime$ with the base point $x_0$, which goes around $v_2^\prime$ anti-clockwise, satisfying that  $\Gamma_2 \cap \rho^\prime$ presents the braid $\sigma_j$.
By applying a CI-move of type (2) around  $v_1$ and $v_2^\prime$, we have a new free edge $F_j$. Thus we have 
$\Gamma_m(b_1, \Delta^n) \cup  F_j.$ See Fig. \ref{fig-0521-01}.
The case when $\epsilon=-1$ can be shown likewise.
\begin{figure}
\includegraphics*{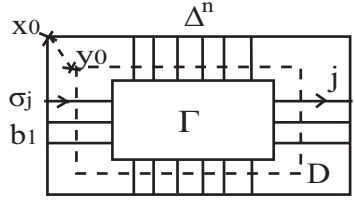}
\caption{An $m$-chart $\Gamma_m(b, \Delta^n)$ presenting $\mathcal{S}_m(b, \Delta^n)$.}
\label{fig-0510-02}
\end{figure}

\begin{figure}
\includegraphics*{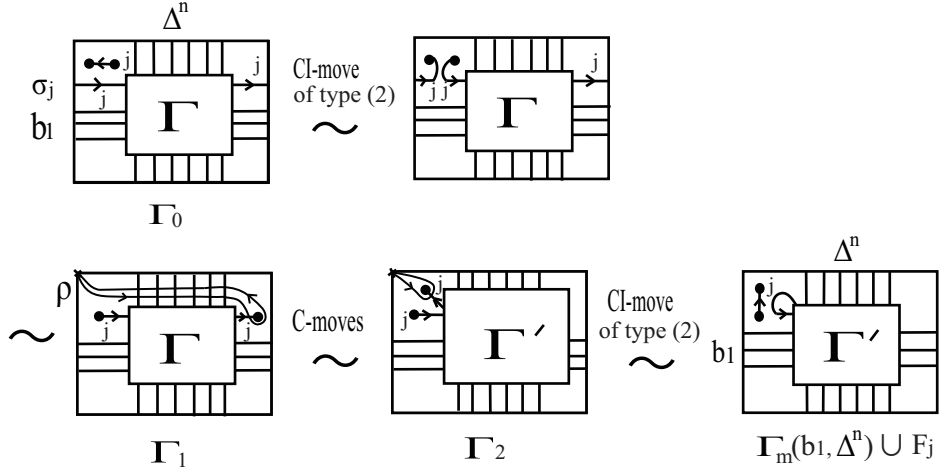}
\caption{An $m$-chart $\Gamma_m(\sigma_j b_1, \Delta^n) \cup F_j$ 
is C-move equivalent to $\Gamma_m(b_1, \Delta^n) \cup F_j$. In the figure, $\Gamma$ and $\Gamma^\prime$ contain no black vertices.}
\label{fig-0521-01}
\end{figure}
\end{proof}

The proof of Lemma \ref{lem2} can be shortened by using the braid monodromies and braid systems of a simple braided surface over a 2-disk (see \cite[Chapter 17]{Kamada02}), as follows. 
We assume that the $m$-charts are on a 2-disk $T-(\mu\cup\lambda)$.
The braid system of $\Gamma_1$ (resp. $\Gamma_m(b_1, \Delta^n) \cup F_j$) is  $ (\sigma_j^{-1}, \Delta^{n} \sigma_j \Delta^{-n})$ (resp. $(\sigma_j^{-1}, \sigma_j)$). 
Since $\Delta^{n} \sigma_j \Delta^{-n} = \sigma_j$ in $B_n$, the braid systems are equivalent. Hence we can see that their presenting braided surfaces are equivalent; thus the $m$-charts are C-move equivalent. See \cite[Chapter 17]{Kamada02}.
 
\begin{lem}\label{lem:0522-1}
We can transform the $m$-chart on a 2-disk as in Fig. \ref{fig-0523-01} (1) to the $m$-chart as in Fig. \ref{fig-0523-01} (2) by C-moves and ambient isotopies.
\end{lem}
\begin{proof}
\begin{figure}
\includegraphics*{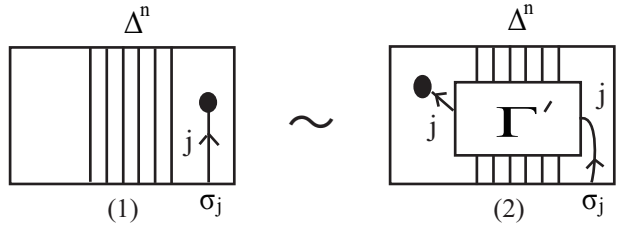}
\caption{Two equivalent $m$-charts, where $\Gamma^\prime$ contains no black vertices.}
\label{fig-0523-01}
\end{figure}
\begin{figure}
\includegraphics*{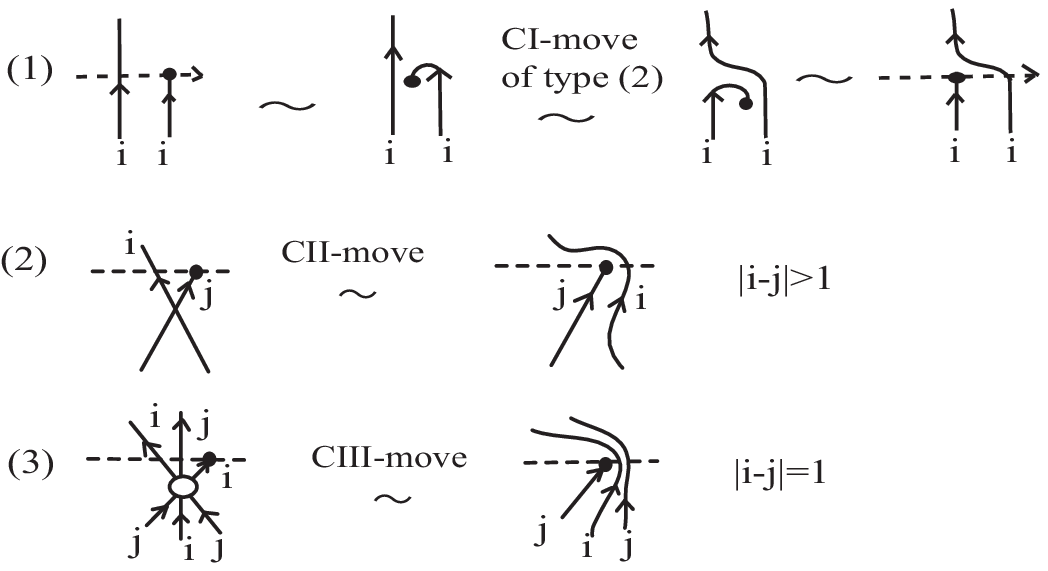}
\caption{}
\label{fig-0522-01}
\end{figure}
Let $\Gamma_1$ be an $m$-chart on a 2-disk with exactly one black vertex. 
Let $\rho$ be a path satisfying that $\Gamma_1\cap \rho$ consists of finite points of edges and the black vertex of $\Gamma_1$ such that $\Gamma_1 \cap\rho$ presents a singular $m$-braid $c \dot{\sigma}_i$ for some $m$-braid presentation $c$ and $i \in \{1,\ldots,m-1\}$, where $\dot{\sigma}_i$ denotes the singular point presented by the black vertex. 
By Fig. \ref{fig-0522-01}, by C-moves and ambient isotopies, we can transform $\Gamma_1$ to an $m$-chart $\Gamma_2$ with $\partial \Gamma_2=\partial \Gamma_1$ such that $\Gamma_2$ has exactly one black vertex which is on $\rho$, and $\Gamma_2\cap\rho$ presents a singular braid related with $c\dot{\sigma}_i$ by the following transformations:
\begin{enumerate}
\item
$\sigma_i\dot{\sigma_i} \leftrightarrow \dot{\sigma_i}\sigma_i$, 

\item
$\sigma_i \dot{\sigma}_{j}\leftrightarrow\dot{\sigma}_j \sigma_i$, 
where $|i-j|>1$, 

\item
$\sigma_i \sigma_j \dot{\sigma_i}\leftrightarrow\dot{\sigma_j} \sigma_i \sigma_j$, where $|i-j|=1$, 
\end{enumerate}
and further $c$ is related by the following transformations:
\begin{enumerate}
\item[(4)]
$\sigma_i \sigma_j = \sigma_j \sigma_i$, where $|i-j|>1$, 

\item[(5)]
$\sigma_i \sigma_j \sigma_i = \sigma_j \sigma_i \sigma_j$, where $|i-j|=1$,
\end{enumerate}
where $i,j \in \{1,\ldots,m-1\}$.

Recall that $\Delta$ has the presentation $\Delta=(\sigma_1 \sigma_2 \cdots \sigma_{m-1})^m$.
Since
$\Delta^n \dot{\sigma_j}$ is transformed to 
$\dot{\sigma_j}\Delta^n$ by these transformations, we can see that the $m$-chart as in Fig. \ref{fig-0523-01} (1) is C-move equivalent to the $m$-chart as in Fig. \ref{fig-0523-01} (2).
\end{proof}
  
 \section{Quandle colorings and quandle cocycle invariants}\label{sec:4}
In Section \ref{sec:5-1}, we review quandle colorings and quandle cocycle invariants \cite{CJKLS, CKS, CKS04}. 
In Section \ref{sec:5-2}, we review Iwakiri's results \cite{Iwakiri} which gives a lower bound of the unknotting number or the triple point cancelling number of a surface knot or a surface link, by using $p$-colorings and quandle cocycle invariants associated with Mochizuki's 3-cocycle.

\subsection 
{}\label{sec:5-1}
A {\it quandle} \cite{Brieskorn, Joyce, Matveev, Takasaki} is a set $X$ with a binary operation $*:X \times X \to X$ satisfying the following:

\begin{enumerate}
\item (Idempotency) For any $x \in X$, $x*x=x$.
\item (Right invertibility) For any $y,z \in X$, there is a unique $x \in X$ such that $x*y=z$.
\item (Right self-distributivity) For any $x,y,z \in X$, $(x*y)*z=(x*z)*(y*z)$.
\end{enumerate}
For any $x \in X$, let $\mathcal{R}_x :X \to X$ be a map defined by $\mathcal{R}_x(y)=y*x$ for $y \in X$. Since $\mathcal{R}_x$ ($x \in X$) is a bijection by Axiom (2), 
the set $\{\mathcal{R}_x \mid x \in X\}$ generates a group called the {\it inner automorphism group} of $X$, which we denote by $\mathrm{Inn}(X)$. A quandle $X$ is said to be {\it connected} if $\mathrm{Inn}(X)$ acts transitively on $X$.

For a classical link $L$ or 
a surface link $F$, we briefly review quandle colorings and quandle cocycle invariants as follows (for details see \cite{CJKLS, CKS, CKS04}).
Let us denote by $D$ the diagram of $L$ or $F$, i.e. 
the image of $L$ or $F$ by a generic projection 
to $\mathbb{R}^2$ or $\mathbb{R}^3$. In order to indicate crossing information of the diagram, we break the under-arc or the under-sheet into two pieces missing the over-arc or the over-sheet. 
Then the diagram is presented by a disjoint union of arcs, or compact surfaces called {\it broken sheets}. 
Let $B(D)$ be the set of such arcs or broken sheets. 
 An {\it $X$-coloring} for a diagram $D$ of $L$ or $F$ 
is a map $C \,:\, B(D) \rightarrow X$ 
   as in Fig. \ref{1001-1}. 
The image by $C$ is called the {\it color}.

\begin{figure}
  \includegraphics*{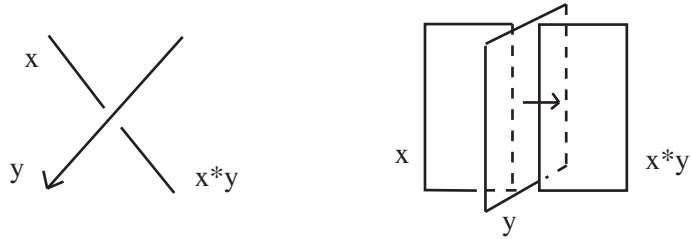}
  \caption{An $X$-coloring $C$, where $x$, $y$, and $x*y$ are the colors of arcs or broken sheets given by $C$, and the orientation of the over-sheet is denoted by a normal.}
  \label{1001-1}
  \end{figure}

Let $G$ be an abelian group. 
A {\it 2-cocycle} with the coefficient group $G$ is a map $f \, :\, X^2 \rightarrow G$ 
satisfying 
\begin{eqnarray*}
& f(s,u)+f(s*u, t*u)=f(s,t)+f(s*t, u), \ \mathrm{and} & \\
& f(s,s)=0 &
\end{eqnarray*}
for any $s,t,u \in X$. 
A {\it 3-cocycle} is a map $f \, :\, X^3 \rightarrow G$ 
satisfying 
\begin{eqnarray*}
& f(s,t,u)+f(s*u, t*u, v)+f(s,u,v)=f(s*t, u,v)+f(s,t,v)+f(s*v, t*v, u*v), \\
& f(s,s,t)=0 \ \mathrm{and} \ f(s, t, t)=0 &
\end{eqnarray*}
for any $s,t,u,v \in X$. 

For an $X$-coloring $C$ of the diagram $D$ of a classical link $L$ or a surface link $F$, the quandle cocycle invariant is defined as follows. For the case of a classical link, at each crossing $r$ of the diagram $D$, the {\it weight} $W_f(r;C)$ at $r$ for a 2-cocycle $f$ is given as in Fig. \ref{0925-3}. 
Put 
\[
\Phi_f(L; C)=\sum_{r \in X_2(D)} W_f(r; C),
\]
where $X_2(D)$ is the set of the crossings of $D$. 
For the case of a surface link, at each triple point $t$ of the diagram $D$, the {\it weight} $W_f(t;C)$ at $t$ for a 3-cocycle $f$ is given as in Fig. \ref{fig-0510-03}. 
Put 
\[
\Phi_f(F; C)=\sum_{t \in X_3(D)} W_f(t; C),
\] 
where $X_3(D)$ is the set of the triple points of $D$. It is known \cite{CJKLS} that $\Phi_{f}(L;C)$ or $\Phi_f(F;C)$ is an invariant of $L$ or $F$ under Reidemeister moves or Roseman moves for quandle colored diagrams. We call it the {\it quandle cocycle invariant} of $L$ or $F$ associated with an $X$-coloring $C$ (see \cite{CJKLS}). Let $X$ be a finite quandle, i.e. a quandle consisting of finitely many elements. Since $B(D)$ is a finite set, so is the set of $X$-colorings for $D$. Let $\mathrm{Col}_X(D)$ be the set of all the $X$-colorings. Then define $\Phi_{f}(L)$ or $\Phi_f(F)$ by the family
  \begin{eqnarray*}
  &&\Phi_{f}(S)=\{\Phi_f(S;C) \mid C \in \mathrm{Col}_X(D) \}, 
\end{eqnarray*}
where $S=L$ or $F$ and $f$ is a 2-cocycle (resp. a 3-cocycle) if $S=L$ (resp. $F$). 
  We call $\Phi_f(L)$ or $\Phi_f(F)$ the {\it quandle cocycle invariant} of $L$ or $F$ 
associated with $f$ \cite{CJKLS}. 
\\

\begin{figure}
\includegraphics*{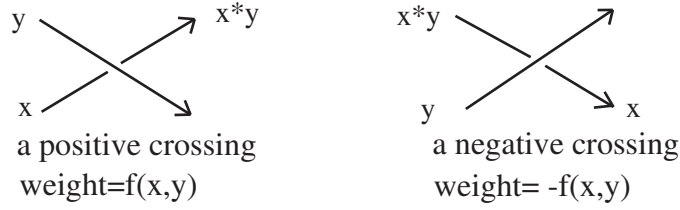}
\caption{The weight at a crossing, where $x$, $y$, and $x*y$ are the colors of arcs by $C$, and $f$ is a 2-cocycle.}
\label{0925-3}
\end{figure}
\begin{figure}
\includegraphics*{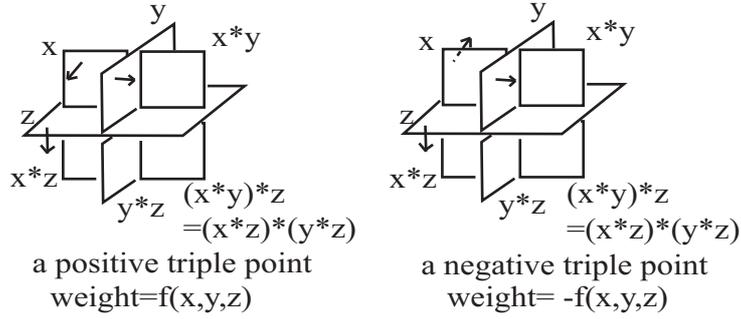}
\caption{The weight at a triple point, where the orientations are denoted by normals, and $x$, $y$, $z$, etc. are the colors by $C$, and $f$ is a 3-cocycle.}
\label{fig-0510-03}
\end{figure}

Next we review the shadow cocycle invariant of a classical link. 
Let $C$ be an $X$-coloring of the diagram $D$ of a classical link $L$. 
For $x \in X$, let $C^*_x$ be a map 
$C^*_x \,:\, B^*(D) \rightarrow X$, where $B^*(D)$ is the union of 
$B(D)$ and the set of regions 
of $\mathbb{R}^2$ separated by the underlying immersed strings of the diagram $D$, satisfying the following conditions: 
\begin{enumerate}
\item
$C^*_x$ restricted to $B(D)$ is coincident with $C$. 
\item 
The color of the regions are as in Fig. \ref{shadow}.   
\item 
The color of the unbounded region is $x$. 
\end{enumerate} 
  By \cite{CKS}, $C^*_x$ exists uniquely for given $C$ and $x$. 
For a 3-cocycle $f$ and $C$ and $x$, let us 
define the weight $W_f^*(r; C, x)$ at a crossing $r$ as in Fig. \ref{shadow}. 
Put 
\[
\Psi_f^*(L;C,x)=\sum_{r \in X_2(D)} W_f^*(r; C, x). 
\]
It is known \cite{CKS} that $\Psi_{f}^*(L;C)$ is an invariant of $L$ under Reidemeister moves or Roseman moves for quandle colored diagrams. We will call $\Psi_f^*(L; C,x)$ the {\it shadow cocycle invariant} of $L$ associated with an $X$-coloring $C$ and the base color $x$ (see \cite{CKS}). 
Then, for a finite quandle $X$, define $\Psi^*_{f}(L)$ by the family
\begin{figure}
\includegraphics*{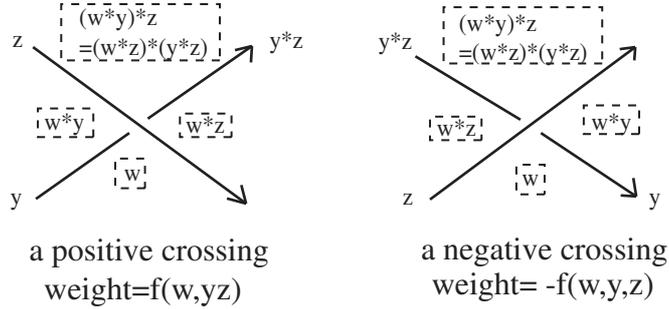}
\caption{The shadow coloring and the weight at a crossing, where $x$, $y$, $w$, etc. are the colors of arcs or regions by $C^*_x$, and $f$ is a 3-cocycle.}
\label{shadow}
\end{figure}
\begin{eqnarray*}
  &&\Psi^*_f(L)=\{\Psi^*_f(L;C,x) \mid C \in \mathrm{Col}_X(D), x \in X \}, 
\end{eqnarray*}
where $f$ is a  3-cocycle. We call $\Psi^*_f(L)$ the {\it shadow quandle cocycle invariant} of $L$ associated with $f$ \cite{CKS}.  
\\

The following lemma is useful to calculate quandle cocycle invariants; see also \cite{IK}. 

\begin{lem}\label{lem:0509-1}
Let $L$ be a classical link, and let $X$ be a quandle.
Let $f$ be a 3-cocycle, and let $C$ be an $X$-coloring of a diagram $D$ of $L$. We denote by $\mathcal{R}_z(C)$ the $X$-coloring $\mathcal{R}_z(C(\cdot)) \,:\, B(D) \rightarrow X$. 
\begin{enumerate}
\item
If $X$ is connected, then, for any $x, y \in X$,
$\Psi_f^*(L;C,x)=\Psi_f^*(L;C,y)$.
\item
For any $x,z \in X$, $\Psi_f^*(L;\mathcal{R}_z(C),x)=\Psi_f^*(L;C,x)$.
\end{enumerate}
\end{lem}

\begin{proof}
\begin{figure}
\includegraphics*{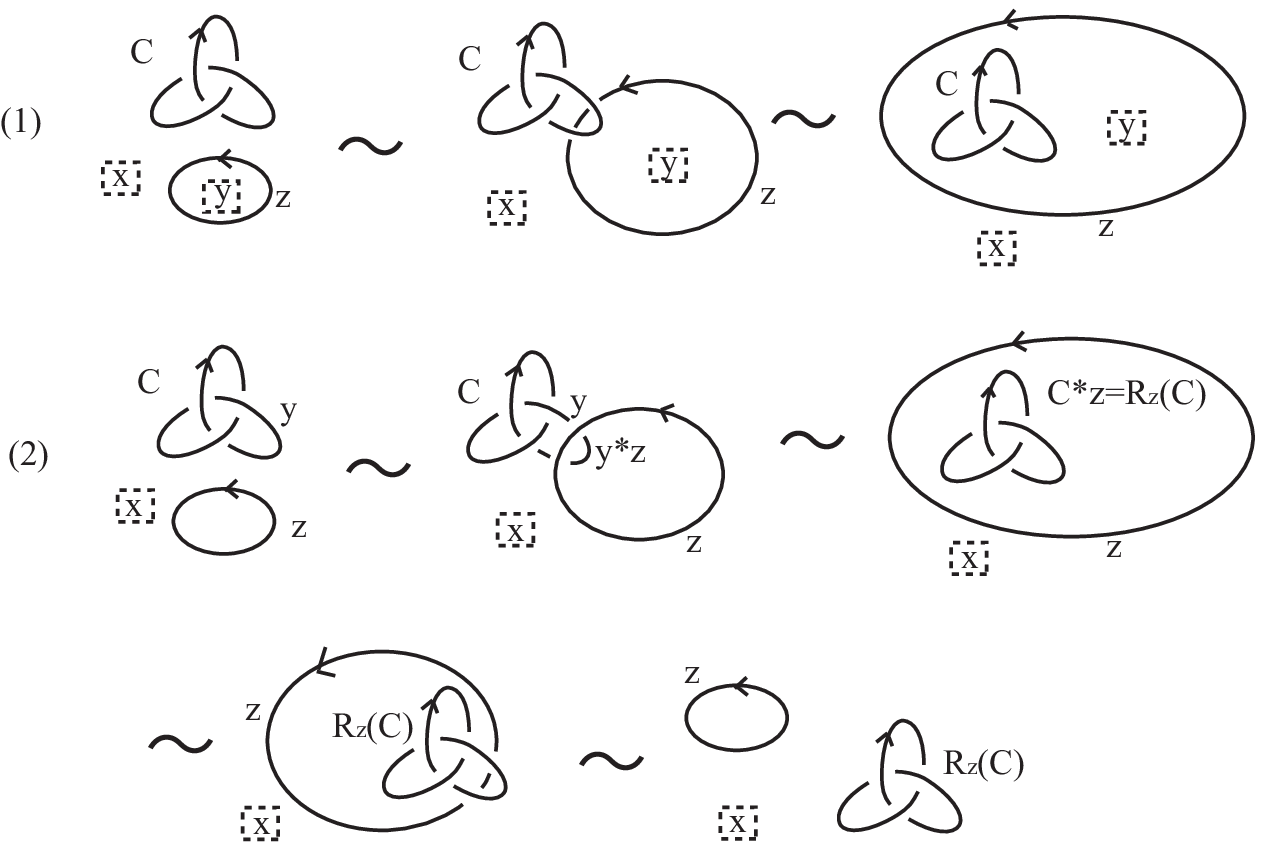}
\caption{}
\label{fig-0510-01}
\end{figure}

We say we move a circle over (resp. under) $D$ by Reidemeister moves if the moves present the transformation of the presented link which moves the trivial knot presented by the circle over (resp. under) $L$ with respect to the hight direction. 

(1) By hypothesis, it suffices to show $\Psi_f^*(L;C,x)=\Psi_f^*(L;C,y)$ for $x, y \in X$ satisfying $\mathcal{R}_z(x)=y$ or  $\mathcal{R}_z^{-1}(x)=y$ for some $z \in X$. 
For the case when $\mathcal{R}_z(x)=y$ (resp. $\mathcal{R}_z^{-1}(x)=y$), 
let us consider the split union of $D$ colored by $C$ and a circle colored by $z$ and oriented anti-clockwise (resp. clockwise), with the base color $x$; note that the region inside the circle is colored by $y$. We denote by $L \cup S^1$ the split union of $L$ and the circle, and we denote by $\tilde{C}$ the associated color. By definition, the circle does not contribute to the shadow cocycle invariant; hence, we see that $\Psi_f^*(L \cup S^1; \tilde{C}, x)=\Psi_f^*(L;C,x)$.  

Then, move the circle under $D$ by Reidemeister moves to the form surrounding $D$ as in Fig. \ref{fig-0510-01} (1). Then, $D$ is colored by $C$, and the region surrounding $D$ is colored by $y$; hence $\Psi_f^*(L \cup S^1; \tilde{C},x)=\Psi_f^*(L;C,y)$.

(2) Let us consider the split union of $D$ colored by $C$ and a circle colored by $z$ and oriented anti-clockwise, with the base color $x$. We use the same notation used in (1). We see that $\Psi_f^*(L \cup S^1; \tilde{C},x)=\Psi_f^*(L;C,x)$.

Then, move the circle over $D$ by Reidemeister moves to the form surrounding $D$, and then move the circle under $D$ by Reidemeister moves to the form of a split union of $D$ and the circle as in Fig. \ref{fig-0510-01} (2) Then, $D$ is colored by $\mathcal{R}_z(C)$, and the base color is $x$; hence $\Psi_f^*(L \cup S^1; \tilde{C},x)=\Psi_f^*(L;\mathcal{R}_z(C),x)$.
\end{proof}

\subsection{}\label{sec:5-2}
Recall that $p$ is an odd prime.
The {\it dihedral quandle} $R_p$ is a set 
$R_p=\mathbb{Z}/p\mathbb{Z}$ with a binary operation $x*y=2y-x$ for any $x,y \in R_p$. 
Note that since $\mathbb{Z}/p\mathbb{Z}$ is a field, $R_p$ is connected. 
 We call an $R_p$-coloring a {\it p-coloring}; see \cite{CJKLS, Fox, Iwakiri}. 
For a classical link $L$ or a surface link $F$, we denote by $\mathrm{Col}_p(D)$ the set of $p$-colorings for a diagram $D$ of $L$ or $F$. 
The number of the elements of $\mathrm{Col}_p(D)$ is an invariant of $L$ or $F$ \cite{CJKLS}, which will be denoted by $|\mathrm{Col}_p(L)|$ or $|\mathrm{Col}_p(F)|$.

\begin{lem}[\cite{Iwakiri}]\label{lem:linear}
For a classical link $L$ or a surface link $F$, 
the set of $p$-colorings $\mathrm{Col}_p(D)$ is a linear space over $\mathbb{Z}/p\mathbb{Z}$, where $D$ is a diagram of $L$ or $F$.
\end{lem}
We briefly review the proof.
We denote by $B(D)$ the set of arcs or broken sheets of $D$, and let $r$ be the number of elements of $B(D)$. 
 We can regard the set of all maps from $B(D)$   to $\mathbb{Z}/p\mathbb{Z}$, as a linear space $\{(x_1, \ldots, x_r) \in (\mathbb{Z}/p\mathbb{Z})^r\}$, 
where $x_i$ denotes the the image of the $i$th arc or broken sheet of $D$ ($i=1,\ldots,r$). 
Let us denote by $s$ the number of crossings or double point curves in $D$. 
Then, $\mathrm{Col}_p(D)=\{(x_1, \ldots, x_r) \in (\mathbb{Z}/p\mathbb{Z})^r \mid f_1=\ldots=f_s=0 \}$, 
where $f_j=2y-x-z$ for $x,y,z \in \{x_1,\ldots,x_r\}$ which are the colors of arcs or broken sheets around the $j$th crossing or double point curve satisfying $x*y=z$ ($j=1,\ldots,s$). 
Hence, $\mathrm{Col}_p(D)$ is isomorphic to 
a linear space $(\mathbb{Z}/p\mathbb{Z})^k$ for some integer $r-s\leq k \leq r$.
\\

Mochizuki \cite{Mochizuki} showed that for any odd prime $p$, the 3-cocycles for $R_p$ with the coefficient group $\mathbb{Z}/p \mathbb{Z}$ form a group isomorphic to $\mathbb{Z}/p \mathbb{Z}$. 
Its generator is reduced (see \cite{Asami-Satoh}) to a map 
given by 
  \[
\theta_p (s,t,u)=(s-t) ( (2u-t)^p+t^p-2u^p )/p \in 
\mathbb{Z}/p \mathbb{Z}. 
\]
We call $\theta_p$ {\it Mochizuki's 3-cocycle}, and we denote the quandle cocycle invariant (resp. shadow cocycle invariant) associated with $\theta_p$ by $\Phi_p(F)$ (resp. $\Psi^*_p(L)$).

\begin{thm}[\cite{Iwakiri}]\label{thm:0516-01}
For a surface knot $F$, 
let $k$ be a positive integer such that $|\mathrm{Col}_p(F)|=p^k$. 
Then $u(F) \geq k-1$.
\end{thm}

For a surface link $F$, 
let us denote by $a_0(\Phi_p(F))$ the number of $0$ in the multi-set $\Phi_p(F)$.
\begin{thm}[\cite{Iwakiri}]\label{thm:0516-02}
For a surface link $F$, let $k$ be a positive integer such that 
$|Col_p(F)|=p^k$, and let $k^\prime$ be a positive integer such that 
$a_0(\Phi_p(F))<p^{k^\prime}$.
Then 
\[
\tau (F) \geq k-k^\prime+1.
\]
\end{thm}
These results are shown by using the facts that 
the set of $p$-colorings can be regarded as a linear space and that a 1-handle addition implies an addition of a new relation to the linear space; see \cite{Iwakiri} (see also the proof of Theorem \ref{thm:0510-1}). In \cite{Iwakiri}, Theorem \ref{thm:0516-02} is shown for the quandle cocycle invariants associated with 3-cocycles valued in a $G_{R_p}$-module, where $G_{R_p}$ denotes the associated group of the dihedral quandle $R_p$. Note that it is known \cite{Nosaka} that any quandle cocycle invariant using the dihedral quandle $R_p$ of prime order is a scalar multiple of the quandle cocycle invariant associated with Mochizuki's 3-cocycle.

\section{Lower bounds of unknotting numbers} \label{sec:6}
In this section, we show Theorems \ref{thm:unknotting} and \ref{thm:=}. 
 
\subsection{Notations and lemmas}\label{sec:6-1}
For a quandle $X$ with a binary operation $*$,
let $\bar{*}: X \times X \to X$ be a binary operation defined by $x \bar{*} y= \mathcal{R}_y^{-1}(x)$ for $x, y \in X$; note that $\bar{*}=*$ if $X=R_p$.

For an $m$-braid $b$, let $A_b : X^m \to X^m$ be a map associated with $b$ by an $X$-coloring, i.e. $A_b$ is a map determined by $A_b=A_{b_1} \circ A_{\sigma_j^{\epsilon}}$ for a presentation $b=\sigma_j^{\epsilon}b_1$ ($j \in \{1,\ldots,m-1\}, \epsilon \in \{+1,-1\}$), where \begin{eqnarray*}
&&A_{\sigma_i}(x_1, \ldots, x_m)=(x_1, \ldots, x_{i-1}, x_{i+1}, x_i*x_{i+1}, x_{i+2}, \ldots,x_m), \text{ and} \\
&&A_{\sigma_i^{-1}}(x_1, \ldots, x_m)=(x_1, \ldots, x_{i-1}, x_{i+1}\bar{*} x_i, x_i, x_{i+2}, \ldots,x_m),  \end{eqnarray*}
for $i=1,\ldots,m-1$.
Remark that $A_b$ is well-defined and bijective.

\begin{ex}
For $R_p$ and a 2-braid $\sigma_1$, 

\[
A_{\sigma_1}=\begin{pmatrix}
0 & 1 \\
-1 & 2
\end{pmatrix}.
\]
 \end{ex}
\begin{lem}\label{lem:0501-1}
For a finite quandle $X$ and any $m$-braid $b$, there is an integer $k$ such that 
$A_b^k=I$, where $I$ denotes the identity map.
\end{lem}

\begin{proof}
Since $X^m$ consists of finite elements, 
for any $\mathbf{x} \in X^m$, there exist integers $k, n$ such that 
$A_b^{n+k}(\mathbf{x})=A_b^{n}(\mathbf{x})$. 
Since the inverse map $A^{-1}_b$exists,  
$A_b^{k}(\mathbf{x})=\mathbf{x}$. 
Let $k^\prime$ be the least common multiple of such $k$ for all $\mathbf{x} \in X^m$. 
Then $A_b^{k^\prime}=I$. 
\end{proof}

In particular, we have the following lemma. 
Recall that $l=2$ (resp. $p$) if $m$ is odd (resp. even). 

\begin{lem}\label{lem:0501-2}
For $R_p$, 
\[
A_\Delta^l=I.
\]
\end{lem}

Recall that $\mathcal{R}_y \,:\, X \rightarrow X$ 
is a map defined by $\mathcal{R}_y(x)=x*y$ for any $x,y \in X$. 
Let $\mathcal{R}_{\emptyset}=\mathrm{id}_X$. For $\mathbf{y}=(y_1, \ldots,y_k) \in X^k$, we denote $\mathcal{R}_{y_k} \circ \cdots \circ\mathcal{R}_{y_2} \circ \mathcal{R}_{y_1}$ by  $\mathcal{R}_{(y_1, y_2, \ldots, y_k)}$, 
and we denote the $i$th iterate of $\mathcal{R}_{\mathbf{y}}$ by $\mathcal{R}_{\mathbf{y}}^i$ for a positive integer $i$, and we denote a cartesian power of $\mathcal{R}_{\mathbf{y}}$ by the same symbol. 
 
For a given $X$-coloring of $\hat{b}$, let $x_i$  be the color of the $i$th initial arc of the $m$-braid $b$ ($i=1,\ldots, m$). Put $\mathbf{x}=(x_1, \ldots, x_m)$.

\begin{proof}[Proof of Lemma \ref{lem:0501-2}]
By Lemma \ref{lem:0503-01}, it suffices to show that $\mathcal{R}^l_{\mathbf{x}}=\mathrm{id}$. 
Note that we calculate in $\mathbb{Z}/p\mathbb{Z}$. 

When $m$ is odd, by calculation, we see that 
\begin{eqnarray*}
\mathcal{R}_{\mathbf{x}}(x) 
&=& 2x_m-2x_{m-1}+2x_{m-2} -\cdots +2x_1-x\\
&=& 2y-x\\
&=& x*y, 
\end{eqnarray*}
where $y=x_m-x_{m-1}+\cdots -x_2+x_1$.
Hence 
$\mathcal{R}^2_{\mathbf{x}}(x)=(x*y)*y =x$; thus $A_\Delta^2=I$.

When $m$ is even, 
\begin{eqnarray*}
\mathcal{R}_{\mathbf{x}}(x)  
&=& 2y+x,
\end{eqnarray*}
where $y=x_m-x_{m-1}+x_{m-2} -\cdots -x_1$. 
Hence
$\mathcal{R}^p_{\mathbf{x}}(x)=2py+x$, which is $x$ modulo $p$; thus $A_{\Delta}^p=I$. 
 \end{proof}
 
\begin{lem}\label{lem:0503-01}
For any positive integer $i$, 
\[
A_{\Delta}^{i}(\mathbf{x})=\mathcal{R}_{\mathbf{x}}^{i}(\mathbf{x}). 
\]
\end{lem}
\begin{proof}
Put $\mathbf{x}_1=\mathbf{x}$, and $\mathbf{x}_j=\mathcal{R}_{\mathbf{x}_{j-1}}(\mathbf{x}_{j-1})$ ($j>1$). For  $\mathbf{z}=(z_1, \ldots, z_k)$ and $\mathbf{z}^\prime=(z^\prime_1, \ldots, z^\prime_{k^\prime})$, we denote $(z_1, \ldots, z_k, z_1^\prime, \ldots, z^\prime_{k^\prime})$ by $\mathbf{z z^\prime}$.
Then, we can see that 
\begin{eqnarray*}
&& A_{\Delta}^{j-1}(\mathbf{x})=\mathbf{x}_{j}, \text{and} \\ 
&& \mathbf{x}_{j}=\mathcal{R}_{\mathbf{x_1 x_2 \cdots x_{j-1}}}(\mathbf{x})=\mathcal{R}_{\mathbf{x}}^{j-1}(\mathbf{x}),  
\end{eqnarray*}
for any integer $j>1$; see \cite[Lemma 5.4]{N}. 
Hence we have the required result.
 \end{proof}

\subsection{Proofs of Theorems \ref{thm:unknotting} and \ref{thm:=}}\label{sec:6-2}
\begin{proof}[Proof of Theorem \ref{thm:unknotting} ]
 
Since $A_{\Delta}^l=I$ by Lemma \ref{lem:0501-2}, 
$|\mathrm{Col}_p(\mathcal{S}_m(b, \Delta^{ln}))|=|\mathrm{Col}_p(\hat{b})|$. Since $\hat{b}$ is a knot, $\mathcal{S}_m(b, \Delta^{ln})$ is a $T^2$-knot. Hence 
the required result follows from Theorem \ref{thm:0516-01}.
\end{proof}
 
 \begin{proof}[Proof of Theorem \ref{thm:=} ]
Since $A_{\sigma_i}^p=I$ ($i=1,\ldots,m-1$), we can see that $A_b=I$, and hence $|\mathrm{Col}_p(\hat{b})|=p^m$; thus Theorems \ref{thm:0503-01} and \ref{thm:unknotting} imply the required result.
 \end{proof}

\subsection{Examples}\label{sec:6-3}
 
(1) For an $m$-braid
$b=\sigma_1^p \sigma_2^p \cdots \sigma_{m-1}^p$ and any integer $n$,
\[
u(\mathcal{S}_m(b, \Delta^{ln})) = m-1.
\] 

(2) For a 3-braid $b=(\sigma_1 \sigma_2^{-1})^4$ and for any integer $n$,
\[
u(\mathcal{S}_3(b, \Delta^{2n}))=2. 
\]

\begin{proof}
(2)
For $R_3$, since $A_{\sigma_1\sigma_2^{-1}}=\begin{pmatrix}
0 & 1 & 0 \\
-2 & 4 & -1\\
-1 & 2 & 0
\end{pmatrix}=
\begin{pmatrix}
0 & 1 & 0 \\
1 & 1 & -1\\
-1 & -1 & 0
\end{pmatrix}$, 
$A_b=A_{\sigma_1\sigma_2^{-1}}^4=\mathrm{I}$.
\end{proof} 
 
\section{Triple point cancelling numbers}
\label{sec:7}
In this section, we show Theorems \ref{thm:0510-3}, \ref{thm:0510-1} and \ref{cor:0510-2}, and Corollary \ref{cor:0122-01}. 
In Section \ref{sec:7-1}, we calculate quandle cocycle invariants associated with Mochizuki's 3-cocycle (Theorem \ref{thm:0505-02}), and we show Theorem \ref{thm:0510-3}. In Section \ref{sec:7-2}, we calculate the quandle cocycle invariants for the torus-covering knots as in Theorem \ref{thm:=} (Theorem \ref{thm:0505-03}), and we show Theorems \ref{thm:0510-1} and \ref{cor:0510-2}, and Corollary \ref{cor:0122-01}. Section \ref{sec:7-3} is devoted to prove lemmas.

\subsection{Calculation of quandle cocycle invariants}\label{sec:7-1}
We use the notations given in Section \ref{sec:6-1}. 
Recall that 
for a given $p$-coloring of $\hat{b}$, 
$x_i$ is the color of the $i$th initial arc of the $m$-braid $b$ ($i=1,\ldots, m$).

\begin{thm}\label{thm:0505-02}
Let $b$ be an $m$-braid and let $n$ be an integer. Assume that $A_{\Delta}^n=I$. 
Then
\[
\Phi_{p}(\mathcal{S}_m(b,\Delta^n))=\{\, -mn\Psi_p^*(\hat{b}; C,0) \mid C \in \mathrm{Col}_p(\hat{b}) \}.
\]
\end{thm}

\begin{proof}
We use the notations given in Section \ref{sec:6-1}. 
Recall that $\mathbf{x}=(x_1, \ldots, x_m)$, and $\mathbf{x}_1=\mathbf{x}$, and $\mathbf{x}_j=\mathcal{R}_{\mathbf{x}_{j-1}}(\mathbf{x}_{j-1})$ ($j>1$). 

Let us consider the case when $n \geq 0$. The assumption $A_{\Delta}^n=I$ implies that for any $C \in \mathrm{Col}_p(\hat{b})$, $\mathcal{R}_\mathbf{y}=\mathrm{id}_{R_p}$, where $\mathbf{y}=\mathbf{x}_1 \mathbf{x}_2 \cdots \mathbf{x}_n$.
Hence, by \cite{N}, the quandle cocycle invariant of $\mathcal{S}_m(b, \Delta^{n})$ associated with Mochizuki's 3-cocycle $\theta_p$ is presented by 
\begin{equation*} 
\Phi_p(\mathcal{S}_m(b, \Delta^{n}))=\{
\Phi_{f_p}(\hat{b}; C) -\sum_{i=1}^m \sum_{j=1}^n 
\Psi_p^*(\hat{b}; \mathcal{R}_{\mathbf{x}}^{j-1} (C), \mathcal{R}_{\mathbf{x}}^{j-1} (x_i)) \mid C \in \mathrm{Col}_p(\hat{b})\}, 
\end{equation*}
where 
$\Phi_{f_p}(\hat{b}; C)$ is the quandle cocycle invariant of $\hat{b}$,  
and $\Psi_p^*(\hat{b}; C, x)$ is the shadow cocycle invariant of $\hat{b}$. 
Here 
$\mathbf{x}$ is determined from $C$ and $b$,  
and $f_p$ is the 2-cocycle determined from $\theta_p$ and $\mathbf{y}=(y_1, \ldots,y_{k})$ by \begin{equation} \label{0901-3}
f_p(s,t)=\sum_{j=1}^k 
\theta_p(\mathcal{R}_{(y_1, \ldots, y_{j-1})}(s), \mathcal{R}_{(y_1, \ldots, y_{j-1})}(t), y_j),  
\end{equation} 
where $k=mn$. 
It is known \cite[Corollary 2.5]{Mochizuki} that $H^2(R_p; \mathbb{Z}/p\mathbb{Z})$ vanishes, and hence we can see that $f_p$ is a coboundary. This implies
that $\Phi_{f_p}(\hat{b};C)=0$ for any $C$ (\cite{CJKLS}). 
 Since $R_p$ is connected,  $\Psi_p^*(\hat{b};C,x)=\Psi_p^*(\hat{b};C,0)$ for any $x \in R_p$ by Lemma \ref{lem:0509-1} (1). 
Further, it follows from Lemma \ref{lem:0509-1} (2) that $\Psi_p^*(\hat{b}; \mathcal{R}_\mathbf{x}(C),0)=\Psi_p^*(\hat{b}; C,0)$. Thus we have the result.

Let us consider the case when $n<0$. 
Since $\Phi_p(\mathcal{S}_m(b, \Delta^{n}); C)+\Phi_p(\mathcal{S}_m(b, \Delta^{-n}); C)
=\Phi_p(\mathcal{S}_m(b, \Delta^{n}\Delta^{-n}); C)=\Phi_p(\mathcal{S}_m(b, e); C)=0$ for any $p$-coloring $C$ (see \cite{N}: in order to calculate $\Phi_p(\mathcal{S}_m(b, \Delta^{n}); C)$, we take the sum of the weights of the triple points presented by the transformation $b \Delta^n \to \Delta^n b$), $\Phi_p(\mathcal{S}_m(b, \Delta^{n}); C)=-\Phi_p(\mathcal{S}_m(b, \Delta^{-n}); C)$ for any $C$; thus we have the required result.
\end{proof}

From now on, we use Theorem \ref{thm:0505-02} for the case when $m$ or $n$ is even. For this case, we can check  $\Phi_{f_p}(\hat{b};C)=0$ for any $C$, by Lemmas \ref{lem:0502-4} and \ref{lem:0505-01}.

\begin{proof}[Proof of Theorem \ref{thm:0510-3}]
By Lemma \ref{lem:0501-2}, we see that $A_{\Delta}^{ln}=I$ for any integer $n$. 
By Theorem \ref{thm:0505-02}, for an $m$-braid $b$ ($m>0$) and an integer $n$, 
\[
\Phi_{p}(\mathcal{S}_m(b,\Delta^{ln}))=\{\, -lmn\Psi_p^*(\hat{b}; C,0) \mid C \in \mathrm{Col}_p(\hat{b}) \}. 
\]
Hence, for a positive odd integer $m$, an $m$-braid $b$ and an integer $n$, 
\[
\Phi_{p}(\mathcal{S}_m(b,\Delta^{2n}))=\{\, -2mn\Psi_p^*(\hat{b}; C,0) \mid C \in \mathrm{Col}_p(\hat{b}) \}. 
\]
Since $\Psi_p^*(\hat{b})$ consists of $p$ copies of $\Psi_p^*(\hat{b}; C,0)$ for each $C$ by Lemma \ref{lem:0509-1} (1), for a positive odd integer $m$ with $m \not \equiv 0 \pmod{p}$ and an integer $n$ with $n \not \equiv 0 \pmod{p}$, it holds true that $a_0(\Phi_{p}(\mathcal{S}_m(b,\Delta^{2n})))=a_0(\Psi_p^*(\hat{b}))/p$.
Hence the required result follows from Theorem \ref{thm:0516-02}.
\end{proof}

\subsection{Proofs of Theorems \ref{thm:0510-1} and \ref{cor:0510-2}, and Corollary \ref{cor:0122-01} }\label{sec:7-2}
\begin{thm}\label{thm:0505-03}
Let $b$ be an $m$-braid as in Theorem \ref{thm:=}.
Then 
\[
\Phi_p(\mathcal{S}_m(b,\Delta^{2n})=\{\, \underbrace{2mn\sum_{i=1}^{m-1} \nu_i y_i^2, \ldots,2mn\sum_{i=1}^{m-1} \nu_i y_i^2}_{\text{$p$ copies}} \mid y_1, \ldots,y_{m-1} \in \mathbb{Z}/p\mathbb{Z} \}.
\]
\end{thm}

\begin{proof}
By the proof of Theorem \ref{thm:0510-3}, we have 
\[
\Phi_p(\mathcal{S}_m(b,\Delta^{2n})) =\{-2mn\Psi_p^*(\hat{b};C,0) \mid C \in \mathrm{Col}_p(\hat{b}) \}.
\]
We calculate $\Psi_p^*(\hat{b};C,0)$. 
It is known \cite{Asami-Satoh} that for an $m$-braid $c_i=\sigma_i^p$, $\Psi_p^*(\hat{c_i}; C, 0)=-(x^\prime_{i}-x^\prime_{i+1})^2$, where 
$x^\prime_{i}, x^\prime_{i+1}$ are the colors of the $i$th and $(i+1)$th initial arcs of $c_i$ ($i=1,\ldots,m-1$).
Since $A_{c_i}=I$, it follows that $\Psi_p^*(\hat{b}; C, 0)=\sum_{i=1}^{m-1} \nu_i \Psi_p^*(\hat{c_i}; C, 0)=-\sum_{i=1}^{m-1} \nu_i(x_i-x_{i+1})^2$; recall that 
$x_{i}$ is the color of the $i$th initial arc of the braid $b$ ($i=1,\ldots,m-1$). 
Note that $A_b=A_{\Delta}^{2n}=I$, and hence each $p$-coloring for a diagram of $\hat{b}$ or $\mathcal{S}_m(b,\Delta^{2n})$ is presented by $(x_1, \ldots, x_m)$.
Thus
\begin{eqnarray*}
\Phi_p(\mathcal{S}_m(b,\Delta^{2n}))&=&\{\, 2mn\sum_{i=1}^{m-1} \nu_i (x_i-x_{i+1})^2 \mid x_1, \ldots,x_m \in \mathbb{Z}/p\mathbb{Z} \}\\
&=&\{\, \underbrace{2mn\sum_{i=1}^{m-1} \nu_i y_i^2, \ldots,2mn\sum_{i=1}^{m-1} \nu_i y_i^2}_{p} \mid y_1, \ldots,y_{m-1} \in \mathbb{Z}/p\mathbb{Z} \}.
\end{eqnarray*}
\end{proof}

\begin{proof}[Proof of Theorem \ref{thm:0510-1}]
Let $k$ be an integer with $0 \leq k < (m-1)/2$. Let us denote by $F$ the torus-covering knot given in this theorem, and let us denote by $F^\prime$ the result of $k$ 1-handle additions to $F$. 

We show that $F^\prime$ is not pseudo-ribbon. 
By Theorem \ref{thm:0505-03}, we see that 
$\Phi_p(F)$ 
consists of copies of $2mn\sum_{i=1}^{m-1} \nu_i y_i^2$, where $y_1,\ldots,y_{m-1} \in \mathbb{Z}/p\mathbb{Z}$.
As we see in Lemma \ref{lem:0518-1}, additions of $k$ 1-handles imply that $y_1, \ldots, y_{m-1}$ satisfy $f_1=\ldots=f_k=0$, where $f_j$ ($j=1,\ldots,k$) is a linear combination of  $y_1, \ldots, y_{m-1}$ over  $\mathbb{Z}/p\mathbb{Z}$. Since $\mathbb{Z}/p\mathbb{Z}$ is a field, by changing the indices $i$ of $y_i$ if necessary, we can assume that $\Phi_p(F^\prime)$ 
consists of copies of $2mn \sum_{i=1}^{m-1} \nu_i y_i^2$, where 
$y_1, \ldots, y_{m-1}$ are elements of $\mathbb{Z}/p\mathbb{Z}$ satisfying 
\begin{equation}\label{eq:0518-4}
\begin{pmatrix} y_1 \cr \vdots \cr y_{k^\prime}
\end{pmatrix}=A
\begin{pmatrix}
y_{k^\prime+1} \cr \vdots \cr y_{m-1},
\end{pmatrix}
\end{equation}
where $k^\prime$ is an integer with $0 \leq k^\prime \leq k$, and $A$ is a $k^\prime \times (m-k^\prime-1)$-matrix over $\mathbb{Z}/p\mathbb{Z}$. Put  $l^\prime=m-k^\prime-1$. 
Put $A=(a_{ij})$.
Let $A^\prime=(\sqrt{\nu_i}a_{ij})$, where $1 \leq i \leq k^\prime, 1 \leq j \leq l^\prime$ and $\sqrt{\nu_i}$ is a non-zero element of $\mathbb{Z}/p\mathbb{Z}$ satisfying $\sqrt{\nu_i}^2=\nu_i$ ($i=1,\ldots,k^\prime$); note that since $\nu_i \in \{ x^2 \mid x \in \mathbb{Z}/p\mathbb{Z} \}-\{0\}$,  $\sqrt{\nu_i}$ exists ($i=1,\ldots,k^\prime$).
Let us denote by $\mathbf{a}_j$ the $j$th column of $A^\prime$ ($j=1,\ldots,l^\prime$).

Assume that $F^\prime$ is pseudo-ribbon. Then the quandle cocycle invariant $\Phi_p(F^\prime)$ consists of copies of zero. 
Since $2mn \not\equiv 0 \pmod{p}$, this implies that
\begin{equation}\label{eq:0518-3}
\sum_{i=1}^{m-1} \nu_i y_i^2=0, 
\end{equation}
where $y_1, \ldots, y_{m-1}$ are elements of $\mathbb{Z}/p\mathbb{Z}$ satisfying (\ref{eq:0518-4}).
This implies that 
\begin{eqnarray}
&& \| \mathbf{a}_j \|^2+\nu_{k^\prime+j}=0, \text{ and} \label{eq:0518-1}\\
&& (\mathbf{a}_j, \mathbf{a}_{j^\prime}) =0,  \label{eq:0518-2}
\end{eqnarray}
for $j,j^\prime \in \{1,\ldots, l^\prime\}$, 
where $\| \cdot \|$ denotes the norm and $(\cdot, \cdot)$ denotes the inner product defined by $(\mathbf{a}, \mathbf{b}) := \mathbf{a}^T \mathbf{b}$ for $\mathbf{a}, \mathbf{b} \in (\mathbb{Z}/p\mathbb{Z})^{k^\prime}$. 
 The equation (\ref{eq:0518-1}) is obtained from (\ref{eq:0518-3}) by substituting $y_{k^\prime+i^\prime}=1$ if $i^\prime=j$ and zero otherwise, where $i^\prime \in \{1,\ldots, l^\prime\}$. By substituting $y_{k^\prime+i^\prime}=1$ if $i^\prime=j$ or $j^\prime$ and zero otherwise ($i^\prime \in \{1,\ldots, l^\prime\}$) to (\ref{eq:0518-3}), together with (\ref{eq:0518-1}) we have $2(\mathbf{a}_j, \mathbf{a}_{j^\prime}) =0$, and hence $2 \not\equiv 0 \pmod{p}$ implies (\ref{eq:0518-2}).
Since $\nu_{k^\prime+j} \neq 0$ ($j=1,\ldots,l^\prime$), it follows from  (\ref{eq:0518-1}) that $\mathbf{a}_j \neq \mathbf{0}$ for $j=1,\ldots,l^\prime$. 
Hence $\mathbf{a}_1, \ldots, \mathbf{a}_{l^\prime}$ are mutually orthogonal non-zero vectors in $(\mathbb{Z}/p\mathbb{Z})^{k^\prime}$ by (\ref{eq:0518-2}). 
Thus $l^\prime \leq k^\prime$. 
However, $l^\prime=m-k^\prime-1\geq m-k-1>(m-1)/2>k\geq k^\prime$. 
This is a contradiction, and hence $F^\prime$ is not pseudo-ribbon. 

We have shown that $F$ cannot be transformed to be pseudo-ribbon by $k$ 1-handle additions for $0 \leq k < (m-1)/2$. This implies that at least $(m-1)/2$ 1-handle additions are necessary to transform $F$ to be pseudo-ribbon; thus $\tau(F)  \geq  (m-1)/2$.
\end{proof}

\begin{rem}\label{rem:1}
Theorem \ref{thm:0510-1} gives a better estimate than Theorem \ref{thm:0510-3}. 
For example, let us consider $F=\mathcal{S}_5(b, \Delta^2)$, where $b=\sigma_1^3 \sigma_2^3 \sigma_3^3 \sigma_4^3$, and let us take $p=3$. Theorem \ref{thm:0510-1} implies that $\tau(F) \geq 2$. 

In order to apply Theorem \ref{thm:0510-3}, we 
calculate $a_0(\Psi_3^*(\hat{b}))$ as follows. 
By the proofs of Theorems \ref{thm:0510-3} and \ref{thm:0505-03}, the quandle cocycle invariant $\Psi_3^*(\hat{b})$ consists of $3^2$ copies of 
$2\sum_{i=1}^{4} y_i^2$, where $y_1, \ldots,y_{4} \in \mathbb{Z}/3\mathbb{Z}$; note that it consists of $3$ copies of the quandle cocycle invariant $\Phi_3(F)$.
Since $\mathbb{Z}/3\mathbb{Z}=\{0, \pm 1\}$, 
if $2\sum_{i=1}^{4} y_i^2=0 \pmod{3}$, then 
(1) $\{y_1^2, \ldots, y_4^2\}=\{0,0,0,0\}$, or 
(2) $\{y_1^2, \ldots, y_4^2\}=\{0,1,1,1\}$.
The number of the 4-tuples $(y_1, \ldots, y_4)$ satisfying (1) (resp. (2)) is one (resp. $4 \cdot 2^3=32$). Thus $a_0(\Psi_3^*(\hat{b}))=3^2(1+32)=3^2 \cdot 33$, and hence $3^5<a_0(\Psi_3^*(\hat{b}))<3^6$. 
Since $|\mathrm{Col}_3(\hat{b})|=3^5$, it follows from Theorem \ref{thm:0510-3} that $\tau(F) \geq 5-6+2=1$, but we cannot induce $\tau(F) \geq 2$ from Theorem \ref{thm:0510-3}.  
\end{rem}

\begin{proof}[Proof of Theorem \ref{cor:0510-2}]
Let us denote by $F$ the torus-covering knot given in this theorem.
We show that $\tau(F) \geq 2$, as follows. By Theorem \ref{thm:0505-03}, $\Phi_p(F)$ consists of copies of $6n\sum_{i=1}^2 \nu_i y_i^2$, where $y_1, y_2 \in \mathbb{Z}/p\mathbb{Z}$. If $\tau(F) \leq 1$, then $F$ is transformed to be pseudo-ribbon by a 1-handle addition, and by the same argument as in the proof of Theorem \ref{thm:0510-1}, we see that then one linear relation $a_1 y_1+a_2 y_2=0$ for some $a_1, a_2 \in \mathbb{Z}/p\mathbb{Z}$ 
induces $6n\sum_{i=1}^2 \nu_i y_i^2=0$, and hence $\sum_{i=1}^2 \nu_i y_i^2=0$, where $y_1, y_2 \in \mathbb{Z}/p\mathbb{Z}$; note the assumption that $p\neq 3$ and $n \not\equiv 0 \pmod{p}$. 
The assumption $p^\prime \neq 2$ implies that if $\sum_{i=1}^2 \nu_i y_i^2=0$, then $y_1=y_2=0$; thus at least two relations are necessary to induce $\sum_{i=1}^2 \nu_i y_i^2=0$. This is a contradiction. Thus $\tau(F) \geq 2$, and hence Theorem \ref{thm:=} implies $\tau(F)=u(F)=2$.
\end{proof} 

\begin{proof}[Proof of Corollary \ref{cor:0122-01}]
By the proof of Theorem \ref{cor:0510-2}, since we assume $\nu_i \in \{j^2 \mid j \in \mathbb{Z}/p \mathbb{Z}\}-\{0\}$ ($i=1,2$), it suffices to show that the congruence $x^2+y^2 \equiv 0 \pmod{p}$ does not have a non-zero solution. 
We use an elementary knowledge of number theory; see \cite{Gauss} and \cite[Chapter 5]{IR}. 
For a positive integer $k$, a {\it quadratic residue} $\pmod{k}$ is any number congruent to a square $\pmod{k}$, and a {\it quadratic nonresidue} $\pmod{k}$ is any number which is not congruent to a square $\pmod{k}$. We treat zero as a special case, and we drop the adjective \lq\lq quadratic" when the context is clear. 
It is known \cite[article 98]{Gauss} that modulo an odd prime, the product of two quadratic residues is a residue, and the product of a residue and a nonresidue is a nonresidue. Hence, in order to prove our corollary, it suffices to show that the congruence  $x^2+1\equiv 0 \pmod{p}$ is not solvable. This is obvious from the first supplement to quadratic reciprocity \cite[article 64]{Gauss}:  
 For an odd prime $q$, the congruence $x^2 \equiv -1 \pmod{q}$ is solvable if and only if $q \equiv 1 \pmod{4}$.  
\end{proof}

\subsection{Lemmas}\label{sec:7-3}
\begin{lem}\label{lem:0502-2}
For $m \geq 4$, 
let $g: (\mathbb{Z}/p\mathbb{Z})^{m-1} \to \mathbb{Z}/p\mathbb{Z}$ be a map defined by 
$g(y_1, \ldots,y_{m-1})=\sum_{i=1}^{m-1} \nu_i y_i^2$, where $\nu_1, \ldots, \nu_{m-1} \in \mathbb{Z}/p\mathbb{Z}$. Let $U_j$ 
be the set of elements of $(\mathbb{Z}/p\mathbb{Z})^{m-1}$ with exactly $j$ non-zero entries ($j=1,\ldots,m-1$). Let $p^\prime$ be the minimum number of $j \in \{1,\ldots,m-1\}$ satisfying  $0 \in g(U_j)$. 
If $\nu_i \neq 0 \pmod{p}$ ($i=1,\ldots,m-1$), then 
$p^\prime=2$ or $3$.
\end{lem}

\begin{proof}
It suffices to show that for fixed $\lambda, \mu, \nu \in \mathbb{Z}/p\mathbb{Z}-\{0\}$, there are  $x,y,z \in \mathbb{Z}/p\mathbb{Z}$ such that at least one of them is not zero and $\lambda x^2+\mu y^2+\nu z^2= 0 \pmod{p}$.

Let us assume that 
$\lambda x^2+\mu y^2+\nu z^2 \neq 0 \pmod{p}$ for any $x,y,z \in (\mathbb{Z}/p\mathbb{Z})^3-\{ \mathbf{0} \}$.
Put $p=2q+1$ for a positive integer $q$.
Since $\mathbb{Z}/p\mathbb{Z}=\{ k \mod{p} \mid k =0, \pm1,  \ldots, \pm q \}$, 
$\{ x^2 \mid x \in \mathbb{Z}/p\mathbb{Z} \}=\{  k^2 \mod{p} \mid k =0, 1,\ldots,q \}$, which will be denoted by $U$.
For $n, n^\prime \in \mathbb{Z}/p\mathbb{Z}$, let us denote 
by $n U+n^\prime$ the subset $\{ n x+n^\prime \mid x \in U\} \subset \mathbb{Z}/p\mathbb{Z}$.

By the assumption, by taking $z=1 \pmod{p}$, we see that $\lambda x^2+\mu y^2+\nu \neq 0 \pmod{p}$ for any $x,y \in \mathbb{Z}/p\mathbb{Z}$.
Hence $(-\lambda U) \cap (\mu U+\nu) =\emptyset$.
Since $-\lambda U$ and $\mu U+\nu$ each consist of $q+1$ elements by Lemma \ref{lem:0502-1}, and $2(q+1)>p$, $(-\lambda U) \cap (\mu U+\nu) \neq \emptyset$. This is a contradiction, and hence we have the required conclusion.

\end{proof}

\begin{lem}\label{lem:0502-1}
We use the same notations as in the proof of Lemma \ref{lem:0502-2}. 
Then, for $n, n^\prime \in \mathbb{Z}/p\mathbb{Z}$ with $n \neq 0 \pmod{p}$, $n U+n^\prime$ consists of $q+1$ elements.
\end{lem}

\begin{proof}
Let us take integers representing $n, n^\prime$, which we denote by the same notations. 
It suffices to show that $n i^2 +n^\prime$ and $n j^2 +n^\prime$ are distinct modulo $p$ for $i,j \in \{0,1,\ldots,q\}$ with $i \neq j$.
We see that $(n i^2 +n^\prime)-(n j^2 +n^\prime)
=n(i+j)(i-j)$. For $i,j \in \{0,1,\ldots,q\}$ with $i \neq j$, since $0 < i+j <2q+1=p$ and 
$-p<-q \leq i-j \leq q<p$, $i+j, i-j$ are not zero modulo $p$. Since $p$ is prime, together with $n \not\equiv 0 \pmod{p}$, we see that $n i^2 +n^\prime$ and $n j^2 +n^\prime$ are distinct modulo $p$ for $i,j \in \{0,1,\ldots,q\}$ with $i \neq j$. Hence we have the conclusion.
\end{proof}

\begin{lem}\label{lem:0502-4}
For $\mathbf{y}=(y_1,\ldots,y_k) \in R_p^k$, 
assume that $\mathcal{R}_\mathbf{y}=\mathrm{id}$ and $k>0$ is even. 
Then the 2-cocycle $f_p(s,t)$ defined by (\ref{0901-3}) is:
\[f_p(s,t)=K(s-t),
\]
where $K$ is a constant integer modulo $p$ (see Remark \ref{lem:0502-3}) defined by
\[
K=-\frac{2}{p}(y_1^p-y_2^p+y_3^p- \ldots -y_k^p).
\]
\end{lem}

\begin{proof}
Put $s_j=\mathcal{R}_{(y_1, \ldots,y_{j-1})}(s)$ and 
$t_j=\mathcal{R}_{(y_1, \ldots,y_{j-1})}(t)$ for a positive integer $j$.
Put $\kappa(t,u)=((t*u)^p+t^p-2u^p)/p$.
By definition of the 2-cocycle, we have
\begin{eqnarray*}
f_p(s,t) &=& \sum_{j=1}^k 
\theta_p(s_j,t_j,y_j)\\
&=& \sum_{j=1}^k (s_j-t_j) \,\kappa(t_j, y_j).
\end{eqnarray*}

We show that $s_j-t_j=(-1)^{j-1}(s-t)$ for a positive integer $j$. 
We see that $s_1-t_1=s-t$. 
Assume that $s_n-t_n=(-1)^{n-1}(s-t)$.
Then $s_{n+1}-t_{n+1}=\mathcal{R}_{y_n}(s_n)-
\mathcal{R}_{y_n}(t_n)=(2y_n-s_n)-(2y_n-t_n)
=-(s_n-t_n)=(-1)^n(s-t)$. Hence, by induction on  $n$, we see that $s_j-t_j=(-1)^{j-1}(s-t)$ for any $j>0$.
 Further, $t_j*y_j=\mathcal{R}_{y_j}(t_j)=t_{j+1}$ for any $j>0$. 
Hence 
\begin{eqnarray*}
f_p(s,t) &=& \sum_{j=1}^k
\theta_p(s_j,t_j,y_j)\\
&=& \frac{1}{p}(s-t) \sum_{j=1}^k (-1)^{j-1}(t_{j+1}^p+t_j^p-2y_j^p).
\end{eqnarray*}
Since $k$ is even and $t_{k+1}=\mathcal{R}_{(y_1, \ldots,y_{k})}(t)=t=t_1$ by the assumption that $\mathcal{R}_\mathbf{y}=\mathrm{id}$, we see that 
\begin{eqnarray*}
f_p(s,t) 
&=& \frac{1}{p}(s-t) \sum_{j=1}^k (-1)^{j-1}(-2y_j^p)\\
&=& K(s-t),
\end{eqnarray*}
where  
\[
K=-\frac{2}{p}(y_1^p-y_2^p+y_3^p- \ldots -y_k^p).
\]
\end{proof}

\begin{rem}\label{lem:0502-3}
In the situation of Lemma \ref{lem:0502-4}, since the 2-cocycle $f_p$ is well-defined and valued in $\mathbb{Z}/p\mathbb{Z}$, $K$ is 
a constant integer modulo $p$. 
\end{rem}
 
 \begin{lem}\label{lem:0505-01}
Let $f : R_p \times R_p \to \mathbb{Z}/p\mathbb{Z}$ be a 2-cocycle defined by $f(s,t)=s-t$.
Then, for any classical link $L$ and any $p$-coloring $C$,
$\Phi_f(L;C)=0$.
\end{lem}

\begin{proof}
Let $D$ be a diagram of $L$, and let us denote by $X_2(D)$ the set of crossings of $D$. 
Let $C$ be a $p$-coloring of $D$.
Let us consider a new graph $D^\prime$ obtained from $D$ by cutting each over-arc at each crossing. For each crossing $r$, there are four arcs in $D^\prime$. For a crossing colored as in the left figure of Fig. \ref{1001-1}, its {\it color} is the pair $(x,y)$. Let us attach a weight $+z$ (resp. $-z$) at each arc if it is oriented toward (resp. from) $r$, where $z$ is its color by $C$. 
Let $W(r;C)$ be the sum of weights of the four arcs in $D^\prime$ around $r$.
Then, we see that 
\begin{equation}\label{eq:1}
\sum_{r \in X_2(D)} W(r;C)=0.
\end{equation}
We can see that $W(r;C)=y+x-y-x*y=2(x-y)$ if $r$ is a positive crossing and $W(r;C)=x*y+y-x-y=-2(x-y)$ if $r$ is a negative crossing, where $(x,y)$ is the color of $r$ by $C$; thus $W(r;C)=2W_f(r;C)$ for any $r$ and $C$. By (\ref{eq:1}), $2\Phi_f(L;C)=0$ for any $C$. Since $p$ is odd prime, $2 \not \equiv 0 \pmod{p}$, and hence $\Phi_f(L;C)=0$ for any $C$. (See also \cite{CKS}.)
\end{proof}

\begin{lem}\label{lem:0518-1}
Let us use the same notations as in the proof of Theorem \ref{thm:0510-1}, and put $y_i=x_i-x_{i+1}$ ($i=1,\ldots,m-1$), where $x_i \in R_p$ is the color of the $i$th initial arc of the diagram of the basis $m$-braid $b$ ($i=1, \ldots,m$).
Then, a 1-handle addition to $F$ implies 
that $y_1, \ldots, y_{m-1}$ satisfy $f=0$, where $f$ is a linear combination of  $y_1, \ldots, y_{m-1}$ over  $\mathbb{Z}/p\mathbb{Z}$.
\end{lem}
\begin{proof}

Let $\tilde{\pi}:\mathbb{R}^4 \to \mathbb{R}^3$ be a generic projection such that $\tilde{\pi}|_{N(T)}=\pi^\prime :N(T) \to I \times T$ (see Section \ref{sec:3}), and let $H_0$ be a 1-handle attaching to $F_0=F$.
Let us denote by $D_0$ the diagram $\tilde{\pi}(F_0)$ with over/under information at each double point curve.
We can transform $F_0$ and $H_0$ by equivalence to $F_1$ and $H_1$ satisfying the following: $\tilde{\pi}(H_1)$ is an embedding of a 3-ball $D^2 \times I$ in $\mathbb{R}^3$, and $\tilde{\pi}(H_1) \cup \tilde{\pi}(F_1)=(D^2 \times \partial I) \cap \tilde{\pi}(F_1)$; see \cite{Boyle88, HoKa79}. 
This implies that a 1-handle addition induces the following condition: 
\begin{equation}\label{eq:0519-1}
C(\alpha)=C(\beta),
\end{equation}
where $\alpha, \beta$ are the sheets of the diagram $D_1=\tilde{\pi}(F_1)$ attaching $\tilde{\pi}(H_1)$, and $C \in \mathrm{Col}_p(D_1)$.

Let us denote by $Y$ the linear space spanned by $\{y_1, \ldots,y_{m-1}\}$ over $\mathbb{Z}/p\mathbb{Z}$ such that each element is expressed as a linear combination.
We show that 
any $p$-coloring for $D_1$ is presented by $(x_1, \ldots,x_m)$, and the color of each sheet can be uniquely written as $x_1+f$ $(f \in Y)$, as follows.  
Any $p$-coloring for $D_0$ is presented by $(x_1, \ldots,x_m)$ (see the proof of Theorem \ref{thm:0505-03}), and the color of each sheet of $D_0$ can be written by a 
well-formed expression constructed by using letters among $x_1, \ldots, x_m$ and the binary operators $*$ and $\bar{*}$. Hence any $p$-coloring for $D_1$ is also presented by $(x_1, \ldots,x_m)$, and it follows from $\bar{*}=*$ that the color of each sheet of $D_1$ is written by an expression constructed by using letters among $x_1, \ldots, x_m$ and the binary operator $*$.
We show that the color of each sheet can be written as $x_1+f$ $(f \in Y)$.
It suffices to show the following: (1) There exists $f_i \in Y$ satisfying $x_i=x_1+f_i$ ($i=1,\ldots,m$), and (2) for any $f,g \in Y$, there exists $h \in Y$ such that $(x_1+f)*(x_1+g)=x_1+h$. We can show (1) by taking $f_i=-\sum_{j=1}^{i-1} y_j$ for $i=1,2,\ldots,m$. Since $(x_1+f)*(x_1+g)=2(x_1+g)-(x_1+f)=x_1+(2g-f)$, we see that (2) holds true.
Since $\mathrm{Col}_p(D_0)$ is isomorphic to $(\mathbb{Z}/p\mathbb{Z})^m$ by the proof of Theorem \ref{thm:0505-03} (see also Lemma \ref{lem:linear}), $\mathrm{Col}_p(D_1)$ is also isomorphic to $(\mathbb{Z}/p\mathbb{Z})^m$; thus the color of each sheet of $D_1$ is uniquely written as a linear combination of $x_1, \ldots,x_m$, and hence, it is uniquely written as $x_1+f$ ($f \in Y$).

Thus, the color of each sheet can be written uniquely as $x_1+f$ $(f \in Y)$, and it follows that
the condition (\ref{eq:0519-1}) implies that $y_1, \ldots, y_{m-1}$ satisfy $f^\prime=0$ for some $f^\prime \in Y$; hence we have the conclusion.

\end{proof}

\section*{Acknowledgements}
The author would like to express her sincere gratitude to  Hiroki Kodama and the Referee for their helpful comments. 
The author is supported by JSPS Research Fellowships for Young Scientists.

\end{document}